\documentclass[final,leqno,onefignum,onetabnum]{siamltex1213}
\usepackage[linesnumbered,ruled,vlined]{algorithm2e}
\usepackage{amsfonts}
\usepackage{amsmath}
\usepackage{array}
\usepackage{cases}
\usepackage{color}

\newcommand{\lvertiii}{{\left\vert\kern-0.25ex\left\vert\kern-0.25ex\left\vert}}
\newcommand{\rvertiii}{{\right\vert\kern-0.25ex\right\vert\kern-0.25ex\right\vert}}
\newtheorem{assume}[theorem]{Assumption}
\newtheorem{remark}[theorem]{Remark}

\title{Practical Newton Methods for Electronic Structure Calculations
\thanks{This work was supported by the National Natural Science
Foundation of China under grants 91730302 and 11671389 and the Key Research Program of Frontier Sciences of the Chinese Academy of Sciences under grant QYZDJ-SSW-SYS010.}}

\author{Xiaoying Dai\footnotemark[2], Liwei Zhang\footnotemark[2],  and Aihui Zhou\footnotemark[2]}

\begin{document}

\maketitle

\renewcommand{\thefootnote}{\fnsymbol{footnote}}

\footnotetext[2]{LSEC, Institute of Computational Mathematics and Scientific/Engineering Computing,
Academy of Mathematics and Systems Science, Chinese Academy of Sciences,  Beijing 100190, China; and School of Mathematical Sciences,
University of Chinese Academy of Sciences, Beijing 100049, China. daixy, zhanglw, azhou@lsec.cc.ac.cn}

\author{
Xiaoying Dai\thanks{LSEC, Institute of Computational Mathematics
and Scientific/Engineering Computing, Academy of Mathematics and
Systems Science, Chinese Academy of Sciences, Beijing 100190, China
(daixy@lsec.cc.ac.cn).} \and  Liwei Zhang\thanks{LSEC, Institute of Computational Mathematics
and Scientific/Engineering Computing, Academy of Mathematics and
Systems Science, Chinese Academy of Sciences, Beijing 100190, China
(zhanglw@lsec.cc.ac.cn).}
\and Aihui Zhou\thanks{LSEC, Institute of Computational Mathematics
and Scientific/Engineering Computing, Academy of Mathematics and
Systems Science, Chinese Academy of Sciences, Beijing 100190, China
(azhou@lsec.cc.ac.cn).}
}

\renewcommand{\thefootnote}{\arabic{footnote}}
\begin{abstract}
  In this paper, we propose and analyze some practical Newton methods for electronic structure calculations. We show the convergence and the local quadratic convergence rate for the Newton method when the Newton search directions are well-obtained. In particular, we investigate some basic implementation issues in determining the search directions and step sizes which ensures the convergence of the subproblem at each iteration and accelerates the algorithm, respectively. It is shown by our numerical experiments that our Newton methods perform better than the existing conjugate gradient method, and the Newton method with the adaptive step size strategy is even more efficient.
\end{abstract}




\section{Introduction}
Electronic structure calculations are the fundamental issues in areas as chemistry, materials science and drug design. It is the many-body Schr\"{o}dinger equation that describes the electron's motion. However, the computational cost for solving such a high dimensional linear eigenvalue problem is extremely expensive, which emphasizes the importance of other equivalent or approximated model. The Kohn-Sham density functional theory (DFT) model \cite{HK, KS} is one of the most widely used model in electronic structure calculations which can be formulated as either a nonlinear eigenvalue problem or an orthogonality constraint minimization problem.

The self consistent field (SCF) iterations is usually applied to solving such a nonlinear eigenvalue problem. There are many implementation issues that should be taken into account. For instance, the density mixing approaches will effect the numerical behaviour much and is thus very important to be chosen \cite{Pu}. In addition, the convergence of SCF iterations is still uncertain although its convergence was proven under the assumptions that the gap between the occupied states and unoccupied states is sufficiently large and the second-order derivatives of the exchange correlation functional are uniformly bounded from above \cite{LWWUY, LWWY, YGM}.


In the context of solving the Kohn-Sham DFT model as a Riemannian manifold constrained minimization problem, there are several progresses in recent years. In \cite{SRNB, ZZWZ}, the gradient type methods are investigated where the negative gradient directions are chosen to be the search directions. It has been shown in \cite{ZZWZ} that the gradient type methods can outperform the SCF iterations in many cases. The authors in \cite{DLZZ} construct a conjugate gradient (CG) method which chooses CG directions as search directions and uses the second order Taylor expansion to get an approximate optimal step size at each iteration. The numerical experiments therein show that the CG method outperforms the gradient type methods significantly in both computational time and stability.

As for general optimization problems with orthogonality constraints, we mention that a first-order algorithm is proposed in \cite{GLCY} which presents some new strategies for orthogonalization and optimizes the objective functional band-by-band. Some parallelizable algorithms are investigated in \cite{DLZxZ, GLY}. We also observe that a Riemannian Newton method with a specific form of retraction is applied to a simple nonlinear eigenvalue problem in \cite{ZBJ} and an adaptive trust region Newton method is proposed and analyzed for Riemannian optimization in \cite{HMWY}.


In most of the optimization methods mentioned above, the Armijo-type backtracking procedure is used to guarantee the convergence. Most recently, an adaptive step size strategy for orthogonality constrained line search methods is proposed in \cite{DZZ} which has been proven to be more efficient than the backtracking-based strategy.

In this paper, we propose a Newton method for electronic structure calculations. We prove the convergence and show the convergence rate of our method under some mild assumptions for many orthogonality preserving strategies including all the widely used ones. To make the Newton method more practical, we study a new perspective to find the Newton search directions and propose a convergent algorithm. We also apply the adaptive step size strategy proposed in \cite{DZZ} to the Newton method and prove the convergence of the Newton method with adaptive step size strategy. We carry out several numerical experiments based on the software package Octopus\footnote[1]{OCTOPUS. http://www.tddft.org/programs/octopus.} for some typical systems including the systems which contain thousands of electrons. These numerical experiments show that the Newton method converges faster than the CG method in \cite{DLZZ} while the latter one was, to our knowledge, the most efficient and stable algorithm for minimizing the Kohn-Sham total energy functional till now.

The rest of this paper is organized as follows: In Section \ref{sec-pre}, we introduce the Kohn-Sham DFT model and some notation that will be used in the rest of the paper.  We study the backtracking-based Newton method including its local convergence as well as convergence rate in Section \ref{sec-alg}. In Section \ref{sec-spbs}, we not only investigate the practical way to solve Newton search direction and apply an adaptive step size strategy to the Newton method, but also prove the convergence of the adaptive algorithm. We report several numerical experiments in Section \ref{sec-ne} to show the advantages of our Newton methods. We then conclude in Section \ref{sec-cln}. Finally we provide some related discussions on the retractions in Appendix \ref{app_retraction}.

\section{Preliminaries} \label{sec-pre}
\subsection{Discretized Kohn-Sham model}
By Kohn-Sham density functional theory, the ground state 
of a system consisting of $M$ nuclei and $N$ electrons can be obtained by solving the
following constrained optimization problem
\begin{equation}\label{Emin}
\begin{split}
&\inf_{U = (u_1, \dots, u_N) \in (H^1(\mathbb{R}^3))^N} \ \ \ \ \ \ \ E_{\textup{KS}}(U) \\
s.t.  &\int_{\mathbb{R}^3} u_iu_j = \delta_{ij},  1\leq i,j\leq N,
\end{split}
\end{equation}
where the Kohn-Sham total energy $E(U)$ is defined as \begin{eqnarray}\label{energy}
E_{\textup{KS}}(U)&=&\frac{1}{2} \int_{\mathbb{R}^3} \sum_{i=1}^N|\nabla u_i(r)|^2 dr +
\frac{1}{2}\int_{\mathbb{R}^3}\int_{\mathbb{R}^3}\frac{\rho(r)\rho(r')}{|r-r'|}drdr' \nonumber\\
&&+\int_{\mathbb{R}^3} \sum_{i=1}^N u_i(r)V_{ext}(r)u_i(r)dr+ \int_{\mathbb{R}^3}
\varepsilon_{xc}(\rho)(r)\rho(r)dr,
\end{eqnarray}
and $u_i\in H^1(\mathbb{R}^3), i = 1, \cdots, N$ are the  Kohn-Sham orbitals.
Here $\rho(r)=\sum \limits_{i=1}^N|u_i(r)|^2$ is the electronic density,
$V_{ext}(r)$ is the external potential generated by the nuclei. The $\varepsilon_{xc}(\rho)(r)$ in the forth
term is the exchange-correlation functional, describing the many-body effects of exchange and
correlation, which is not known explicitly, and some approximation (such as local density approximation (LDA),
generalized gradient approximation (GGA))
has to be used \cite{Ma}.

We may discretize the Kohn-Sham model by the plane wave method, the local basis set method, or
the real space methods. Under some proper discretization, the associate discretized Kohn-Sham energy model
can be formulated as
\begin{equation}\label{dis-emin}
\min_{U\in \mathcal{M}^N_{N_g}} \ \ \ E_{\textup{d-KS}}(U),
\end{equation}
where $N_g$ presents the degree of freedom and $\mathcal{M}^N_{N_g}$ is the Stiefel manifold defined by
\begin{equation*}
\mathcal{M}^N_{N_g} = \{U \in\mathbb{R}^{N_g\times N}:   U^TU  = I_N\},
\end{equation*}
with $U=(u_1, u_2, \cdots, u_N)\in \mathbb{R}^{N_g\times N}$ and $I_N$ denotes the identity matrix of order $N$. Typically, $N_g \gg N$. The column vectors of $U$ can be viewed as the
 discretized wave functions, and $E(U)$, the discretized Kohn-Sham total energy, is often called the total energy functional of $U$. If we denote
the charge density by
\begin{equation}\label{density}
\rho(U)=\textup{diag}(UU^T),
\end{equation}
where $\textup{diag}(A)$ is a column vector consisting of the diagonal entries of the matrix $A$, then
the discretized Kohn-Sham total energy can be stated as
 \begin{equation}\label{E-KS}
E_{\textup{d-KS}}(U)=\frac{1}{2}\textup{tr}(U^TLU)+\textup{tr}( U^T V_{ext} U)+\frac{1}{2}\textup{tr}(\rho(U)^TL^\dagger\rho(U)) + \textup{tr}(\rho(U)^T
       \varepsilon_{xc}(\rho(U))),
\end{equation}
where $L\in\mathbb{R}^{N_g\times N_g}$ is the disctetized Laplace operator, $V_{ext}\in\mathbb{R}^{N_g\times N_g}$ is the discretized external potential, $L^\dagger$ is the generalized inverse of $L$, and $\varepsilon_{xc}(\rho(U))$ is the discretized exchange correlation potential. Hereafter, we omit the subscript ``d-KS", i.e., denote $E(U) = E_{\textup{d-KS}}(U)$, for simplicity.


We see that the first order optimal condition of \eqref{dis-emin} is

\begin{eqnarray}\label{dis-evp}
\left\{ \begin{array}{rcl}
\mathcal{H}(U)U&=&U\Lambda, \\[0.2cm]
U^TU&=&I_N,
\end{array}
\right.
\end{eqnarray}
where $\Lambda\in\mathbb    {R}^{N\times N}$ is the Lagrangian multiplier and is symmetric,
\begin{equation}\label{dis-Ham}
\mathcal{H}(U)=\frac{1}{2}L+V_{ext}+\textup{Diag}(L^\dagger\rho(U))+\textup{Diag}(v_{xc}(\rho(U))),
\end{equation}
with $\textup{Diag}(u)$ denotes the diagonal matrix with $u$ on its diagonal, and
\begin{equation*}
v_{xc}(\rho(U))
= \frac{\delta(\rho\varepsilon_{xc}(\rho))}{\delta\rho}.
\end{equation*}

Note that
\begin{eqnarray}\label{orth-invar}
E(U)=E(UP), \forall P\in\mathcal{O}^N,
\end{eqnarray}
where $\mathcal{O}^N$ is the set of all orthogonal matrices of order $N$.
We see that the solution of \eqref{dis-emin} is not unique.
Due to (\ref{orth-invar}), we introduce a Grassmann manifold,
which is the quotient space of the Stiefel manifold and is defined as follows
\begin{equation*}
\mathcal{G}^N_{N_g} = \mathcal{M}^N_{N_g}/\sim.
\end{equation*}
Here, $\sim$ denotes the equivalence relation: we say  $\hat{U} \sim U$,  if there exists $P \in \mathcal{O}^{N}$ such that $\hat{U} = UP$. For any $U \in \mathcal{M}^N_{N_g}$, we denote the equivalence class  by $[U]$, that is,
\begin{equation*}
[U] = \{U P: P \in  \mathcal{O}^{N}\}.
\end{equation*}

In addition, the definition of tangent space at $[U] \in \mathcal{G}^N_{N_g}$ and tangent bundle on Grassmann manifold $\mathcal{G}^N_{N_g}$ are given in \cite{SRNB} that
\begin{equation}\label{dis-tangent}
\mathcal{T}_{[U]}\mathcal{G}^N_{N_g} = \{W\in \mathbb{R}^{N_g\times N}:
W^TU   = 0\},
\end{equation}
and
\begin{equation*}
\displaystyle \mathcal{T}\mathcal{G}^N_{N_g} = \bigcup_{[U]\in\mathcal{G}^N_{N_g}}\mathcal{T}_{[U]}\mathcal{G}^N_{N_g}.
\end{equation*}

To get rid of the non-uniqueness, we turn to study
\begin{equation}\label{emin-Grass}
\min_{[U] \in \mathcal{G}^N_{N_g}} \ \ \ E(U).
\end{equation}
It is worth mentioning that even if we investigate \eqref{emin-Grass} on the Grassmann manifold, the uniqueness of its solution is still unknown since the total energy functional is non-convex. Here and hereafter, we assume that the discretized energy $E(U)$ is smooth enough unless we mention explicitly.

We denote $\nabla E(U)=\mathcal{H}(U)U$ the gradient of $E(U)$ in $\mathbb{R}^{N_g\times N}$, and can calculate the second order derivative of $E$ at point $U$ as
$\nabla^2E(U)\in \mathcal{L}(\mathbb{R}^{N_g\times N}, \mathbb{R}^{N_g\times N})$:
$$\nabla^2E(U)[D]= \mathcal{H}(U)D + 2 \textup{Diag}(J\textup{diag}(DU^T))U, ~U,D\in\mathbb{R}^{N_g\times N},$$ with $J=L^\dagger + \frac{\delta^2(\rho\varepsilon_{xc}(\rho))}{\delta\rho^2}$ \cite{LWWY, WMUZ}.
Similarly, the third order derivative  of $E$ at point $U$ is denoted by $\nabla^3E(U)\in \mathcal{L}(\mathbb{R}^{N_g\times N}\times\mathbb{R}^{N_g\times N},\mathbb{R}^{N_g\times N})$ which satisfies
\begin{eqnarray*}
\nabla^3E(U)[D_1,D_2] &=&2\big(\textup{Diag}(J\textup{diag}(D_2U^T))D_1 \\
&+&\textup{Diag}(J\textup{diag}(D_1U^T))D_2 +\textup{Diag}(J\textup{diag}(D_1D_2^T))U\big) \\
&+&4\textup{Diag}(\textup{Diag}(\frac{\delta^3(\rho\varepsilon_{xc}(\rho))}{\delta\rho^3}\textup{diag}(D_2U^T))\textup{diag}(D_1U^T))U.
\end{eqnarray*}

We see from \cite{EAS} that the gradient of $E(U)$ at
$[U]$ on Grassmann manifold $\mathcal{G}^N_{N_g}$ is a tangent vector in $\mathcal{T}_{[U]}\mathcal{G}^N_{N_g}$ which has the form
\begin{equation*}
\begin{split}
    \nabla_G E(U) & = (I - UU^T)\nabla E(U), ~\forall U\in\mathcal{M}^N_{N_g},
\end{split}
\end{equation*}
or
\begin{equation}\label{gra}
\nabla_G E(U)= \nabla E(U) - U\Sigma = \mathcal{H}(U)U - U\Sigma,
\end{equation}
where $\Sigma = U^T\nabla E(U)= U^T\mathcal{H}(U)U$ is symmetric since
$\mathcal{H}(U)$ is a symmetric operator. We may review $\nabla_G E(U)$ as a $(0, 1)$ type tensor in the sense of
$$\nabla_G E(U)[D] = \langle\nabla_G E(U), D\rangle = \textup{tr}(\nabla_G E(U)^T D), ~\forall D\in\mathcal{T}_{[U]}\mathcal{G}^N_{N_g}.$$
Here, $\langle\cdot,\cdot\rangle$ forms the inner product in $R^{N_g\times N}$.

We also need the  Hessian of $E(U)$ on the Grassmann manifold, which is defined as \cite{EAS}
\begin{equation}\label{hes}
\nabla^2_GE(U)[D] =
(I-UU^T)\nabla^2E(U)[D] - D\Sigma, \forall U\in \mathcal{M}^N_{N_g}, ~D \in
\mathcal{T}_{[U]}\mathcal{G}^N_{N_g},
\end{equation}
we see that $\nabla^2_GE(U)$ can be viewed as a $(0, 2)$ type tensor and we sometime abuse the notation that
\begin{eqnarray*}
\nabla^2_GE(U)[D_1,D_2] &=& \langle\nabla_G^2E(U)[D_1],D_2 \rangle \\
 &=& \textup{tr}(D_2^T\nabla^2E(U)[D_1]) - \textup{tr}(D_2^TD_1\Sigma), \forall \ D_1,D_2 \in
\mathcal{T}_{[U]}\mathcal{G}^N_{N_g}.
\end{eqnarray*}



We obtain from $\nabla^2 E(U)\in\mathcal{L}(\mathbb{R}^{N_g\times N}, \mathbb{R}^{N_g\times N})$ that for all $U\in\mathcal{M}_{N_g}^N$, there exists a constant $\tilde{C}>0$, such that
\begin{equation}\label{upb1}
\|\nabla^2E(U)[D]\|_F\le \tilde{C}\|D\|_F,
\end{equation}
where $\|\cdot\|_F$ denotes the Frobenius norm of a matrix. Thus,
\begin{eqnarray}\label{upb}
\|\nabla^2_GE(U)[D]\|_F&\leq&\|(I-UU^T)\nabla^2E(U)[D]\|_F+\|D\Sigma\|_F \\ \notag
&\leq&\|\nabla^2E(U)[D]\|_F+\|D\|_F\|\Sigma\|_2 \\ \notag
&\leq&\tilde{C}\|D\|_F+|\lambda_{max}\big(\mathcal{H}(U)\big)|\|D\|_F \\ \notag
&\leq&C\|D\|_F, \forall \ D \in
\mathcal{T}_{[U]}\mathcal{G}^N_{N_g}.
\end{eqnarray}
where $C=\tilde{C}+|\lambda_{max}\big(\mathcal{H}(U)\big)|$ is a positive constant.

Similarly, we can calculate the third order derivative of $E$ on the Grassmann manifold as
\begin{eqnarray*}
\nabla_G^3E(U)[D_1,D_2]&=&(I_N-UU^T)\nabla^3E(U)[D_1,D_2]-D_2U^T\nabla^2 E(U)[D_1] \\
&&-D_1U^T\nabla^2 E(U)[D_2]-D_1D_2^T\nabla E(U), \ \forall \ D_1,D_2\in\mathcal{T}_{[U]}\mathcal{G}^N_{N_g},
\end{eqnarray*}
Here, $\nabla_G^3E(U)$ is equivalent to a $(0, 3)$ type tensor, and we hence abuse the notation that
\begin{eqnarray*}
\nabla^3_GE(U)[D_1,D_2,D_3] &=& \langle\nabla_G^3E(U)[D_1,D_2],D_3 \rangle \\
&=&\textup{tr}(D_3^T\nabla^3E(U)[D_1,D_2]) - \textup{tr}(D_3^TD_2U^T\nabla^2E(U)[D_1]) \\
&&- \textup{tr}(D_3^TD_1U^T\nabla^2E(U)[D_2])-\textup{tr}(D_3^TD_1D_2^T\nabla E(U)), \\
&&\forall \ D_1,D_2, D_3 \in \mathcal{T}_{[U]}\mathcal{G}^N_{N_g}.
\end{eqnarray*}

We see that $\nabla_G^3E(U)$ is bounded above by using the same strategy as the Hessian operator, namely, there exists a positive constant which is also denoted by $C$, such that
\begin{equation}\label{3upb}
\|\nabla^3_GE(U)[D_1,D_2]\|_F\leq C\|D_1\|_F\|D_2\|_F, \forall \ D_1, D_2 \in
\mathcal{T}_{[U]}\mathcal{G}^N_{N_g}.
\end{equation}

\subsection{Manifold related}

To address the numerical theory, we introduce two distances on Grassmann manifold $\mathcal{G}^N_{N_g}$. Let $[U], [V] \in \mathcal{G}^N_{N_g}$, with
$U,V\in \mathcal{M}^N_{N_g}$. If $U^TV=A\cos{\Theta}B^T$ and $V-U(U^TV)=A_2\sin{\Theta}B^T$ is the Singular Value Decomposition (SVD) of $U^TV$ and $V-U(U^TV)$ respectively, then we obtain from Lemma A.1 of \cite{DLZZ} that there exists a geodesic
\begin{equation}\label{geodesic1}
 [\Gamma(t)]=[UA\cos{(\Theta t)}A^T+A_2\sin{(\Theta t)}A^T], t\in[0,1],
\end{equation}
such that
\begin{eqnarray*}
[\Gamma(0)]=[U], [\Gamma(1)]=[V].
\end{eqnarray*}
Here, $$\Theta=\textup{Diag}(\theta_1, \theta_2, \cdots, \theta_N)$$ with $\theta_i\in[0, \frac{\pi}{2}]$ being a diagonal matrix, $$\sin{(\Theta t)}=\textup{Diag}(\sin(\theta_1 t), \sin(\theta_2 t), \cdots, \sin(\theta_N t)),$$
and
$$\cos{(\Theta t)}=\textup{Diag}(\cos(\theta_1 t), \cos(\theta_2 t), \cdots, \cos(\theta_N t)).$$

We define the distance between $[U]$ and $[V]$ on the Grassmann manifold by
\begin{equation}\label{dist}
\textup{dist}_{F}([U], [V]) = \min_{P\in \mathcal{O}^{N\times N}} \| U - V P \|_F,
\end{equation}
or
\begin{equation}
\textup{dist}_{geo}([U], [V]) =  \| A_2\Theta A^T \|_F.
\end{equation}

\begin{remark}
It can be verified that
\begin{equation*}
\begin{split}
&\textup{dist}_{F}([U], [V]) = \|2\sin{\frac{\Theta}{2}}\|_F, \\
&\textup{dist}_{geo}([U], [V]) = \| \Theta \|_F,
\end{split}
\end{equation*}
which indicates that these two kinds of distance are equivalent, namely, $$\textup{dist}_{F}([U], [V])\leq\textup{dist}_{geo}([U], [V])\leq2\textup{dist}_{F}([U], [V]).$$
\end{remark}



To avoid the confusion caused by arc and major arc, we define all elements in the tangent bundle which denote the distance of some $[U], [V]\in\mathcal{G}_{N_g}^N$ by

$$\mathcal{D} = \{D\in \mathcal{T}\mathcal{G}_{N_g}^N|\exists [U],[V]\in\mathcal{G}_{N_g}^N, \|D\|_F=\textup{dist}_{geo}([U],[V])\},$$
and it restriction on $\mathcal{T}_{[U]}\mathcal{G}_{N_g}^N$ by
$$\mathcal{D}_{[U]} = \{D\in \mathcal{T}_{[U]}G_{N_g}^N|\exists [V]\in\mathcal{G}_{N_g}^N, \|D\|_F=\textup{dist}_{geo}([U],[V])\}.$$
Then we see that $\mathcal{D}$ and $\mathcal{D}_{[U]}, \forall U\in\mathcal{G}_{N_g}^N$ are bounded and have the following observation.

\begin{remark}
For any $U\in\mathcal{M}^N_{N_g}, D\in\mathcal{D}_{[U]}$, if $D=ASB^T$ is the SVD of $D$, then there exists an unique geodesic
\begin{equation}\label{geodesic3}
 [\Gamma(t)]=[UB\cos{(S t)}B^T+A\sin{(S t)}B^T] 
\end{equation}
starting from $[U]$ and along with direction $D$. We see that \eqref{geodesic1} is nothing but a special case with direction $D=A_2\Theta A^T$.
\end{remark}

Hereafter, we use macro
\begin{equation}\label{geodesic2}
\exp_{[U]}(t D) := [\Gamma(t)]
\end{equation}
to denote such a geodesic on $\mathcal{G}^N_{N_g}$. 
We now define the parallel mapping which maps a tangent vector along the geodesic $\exp_{[U]}(t D)$ \cite{EAS}.

\begin{definition}
The parallel mapping $\tau_{tD}: \ \mathcal{T}_{[U]}\mathcal{G}^N_{N_g}\to\mathcal{T}_{[exp_{[U]}(tD)]}\mathcal{G}^N_{N_g}$ along $\exp_{[U]}(tD)$ is defined as
\begin{equation*}
\tau_{tD}(\tilde{D})=\big((-U\sin{(St)}+A\cos{(St)}A^T+(I_N-AA^T)\big)\tilde{D},
\end{equation*}
where $D=ASB^T$ is the SVD of $D$.
\end{definition}

It can be verified that $\tau_{tD}\tilde{D}\in\mathcal{T}_{[\exp_{[U]}(tD)]}\mathcal{G}^N_{N_g}$ and
\begin{eqnarray}
&&\|\tau_{tD}(D_1)\|_F=\|D_1\|_F, \\
&&\textup{tr}({\tau_{tD} (D_1)}^T\tau_{tD} (D_2))=\textup{tr}(D_1^TD_2), \forall D_1,D_2\in\mathcal{T}_{[U]}\mathcal{G}^N_{N_g}.
\end{eqnarray}

We have the following proposition for our convergence proof from Remark 3.2 and Remark 4.2 of \cite{Smith94}.
\begin{proposition}\label{Taylor}
For $U\in\mathcal{M}^N_{N_g}$, $D\in\mathcal{T}_{[U]}\mathcal{G}^N_{N_g}$, $t\in(0,\frac{1}{\|D\|_F})$, there exist $\xi_i\in(0,t), i=1,2$ such that
\begin{eqnarray}
\label{taylor1} &&E(\exp_{[U]}(tD)) \notag \\
=&&E(U)+t\langle\nabla_GE(U),D\rangle + \frac{t^2}{2}\nabla^2_GE(U)[D, D] \\ \notag
&+&\frac{t^3}{6}\nabla^3_GE(\exp_{[U]}(\xi_1 D))[\tau_{\xi_1D} (D),\tau_{\xi_1D} (D),\tau_{\xi_1D} (D)],
\end{eqnarray}
and
\begin{eqnarray}
\label{taylor2} &&\langle\tau_{tD}^{-1}\nabla_GE(\exp_{[U]}(tD)), \cdot\rangle  \notag\\
&=&\langle\nabla_GE(U), \cdot\rangle +t\nabla^2_GE(U)[D, \cdot] \\
&&+ \frac{t^2}{2}\nabla_G^3E(\exp_{[U]}(\xi_2 D))[\tau_{\xi_2D} (D), \tau_{\xi_2D} (D), \tau_{\xi_2D}(\cdot)]. \notag
\end{eqnarray}
If the energy functional $E(U)$ is of second order differentiable only, then there holds
\begin{eqnarray}\label{taylor3}
&&E(\exp_{[U]}(tD)) \notag \\
=&&E(U)+t\langle\nabla_GE(U),D\rangle \\
&+&t\int_0^t(1-\frac{s}{t})\nabla_G^2 E(\exp_{[U]}(sD))[\tau_{sD}(D),\tau_{sD}(D)]ds, \notag
\end{eqnarray}
and
\begin{eqnarray}\label{taylor4}
&&\langle\tau_{tD}^{-1}\nabla_GE(\exp_{[U]}(tD)), \cdot\rangle  \notag\\
&=&\langle\nabla_GE(U), \cdot\rangle +\int_0^t \nabla^2_GE(\exp_{[U]}(sD))[\tau_{sD}(D), \tau_{sD}(\cdot)]ds. 
\end{eqnarray}
\end{proposition}

Now, we introduce some assumptions that will be used in our analysis in Section \ref{sec-alg}.
First, we need the following assumption.
\begin{assume}\label{assum_lip}
The gradient of the energy functional is Lipschitz continuous. That is, there exists $L_0>0$
such that
\begin{equation*}
\| \nabla E(U) - \nabla E(V) \|_F \leq L_0\|  U - V \|_F, \ \ \forall \ U, V
\in \mathcal{M}^N_{N_g}.
\end{equation*}
\end{assume}

Note that the same assumption is used and discussed in, for instance, \cite{DLZZ, LWWUY, UWYKL}.
From Assumption \ref{assum_lip}, there is a constant $C_0>0$, such that
\begin{equation}\label{der-bound}
\| \nabla E(\Psi) \|_F \leq C_0, \ \ \forall \ \Psi \in \mathcal{M}^N_{N_g},
\end{equation}
which implies
\begin{equation} \label{lip}
\| \nabla_G E(U) - \nabla_G E(V) \|_F \leq L_1\| U - V \|_F, \ \ \forall \ U, V
\in \mathcal{M}^N_{N_g},
\end{equation}
where $L_1 = L_0 + 2C_0$. That is, the gradient of the energy functional on the
Grassmann manifold is Lipschitz continuous, too.

In addition, we also assume that there exists
a local minimizer $[U^\ast]$ of \eqref{dis-emin}, on which the following assumption will be
imposed.
\begin{assume}\label{assum_hess}
There exists $\delta_1>0$, such that $\forall \ [U] \in B([U^\ast],\delta_1)$,
\begin{eqnarray}\label{pos-bnd}
  \nabla^2_GE(U)[D, D] &\ge& \nu_1\| D \|_F^2, \forall \ D \in \mathcal{T}_{[U]}\mathcal{G}^N_{N_g}\notag
\end{eqnarray}
here $\nu_1>0$ is a constant, and
\begin{equation*}
B([U],\delta) := \{[V]\in \mathcal{G}_{N_g}^N: \textup{dist}_{F}([V],
  [U]) \le \delta\}.
\end{equation*}
\end{assume}

Assumption \ref{assum_hess} typically leads to the following lemma:

\begin{lemma}\label{distbound}
Let Assumption \ref{assum_hess} hold true. If $[U^\ast]$ is a local minimizer of \eqref{dis-emin}, then for all $[U]\in B([U^\ast],\delta_1)$, there holds,
\begin{eqnarray}
\frac{\nu_1}{2} \textup{dist}_{geo}([U],[U^\ast])^2\leq E(U)&-&E(U^\ast)\leq \frac{C}{2} \textup{dist}_{geo}([U],[U^\ast])^2, \\
\|\nabla_G E(U)\|_F&\geq& \nu_1 \textup{dist}_{geo}([U],[U^\ast]).
\end{eqnarray}
\end{lemma}

\begin{proof}
For all $[U]\in B([U^\ast],\delta_1)$, there exists an unique geodesic $\exp_{[U^\ast]}(tD)$ such that $\exp_{[U^\ast]}(D)=U$ and $\textup{dist}_{geo}([U], [U^\ast])=\|D\|_F$.
By using Proposition \ref{Taylor} and the fact that $\|\nabla_G E(U^\ast)\|_F=0$, we get
\begin{eqnarray*}
E(U)-E(U^\ast)&=&\frac{1}{2}\nabla^2_GE(\exp_{[U^\ast]}(\xi_1 D))[\tau_{\xi_1D} (D),\tau_{\xi_1D} (D)], \\
\langle\tau_1^{-1}\nabla_GE(U), \cdot\rangle&=&\nabla^2_GE(\exp_{[U^\ast]}(\xi_2 D))[\tau_{\xi_2D}(D), \tau_{\xi_2D}(\cdot)].
\end{eqnarray*}
Thus,
$$\|\nabla_GE(U)\|_F=\|\tau_1^{-1}\nabla_GE(U)\|_F\geq\frac{\langle\tau_1^{-1}\nabla_GE(U), D\rangle}{\|D\|_F}=\frac{\nabla^2_GE(\exp_{[U^\ast]}(\xi_2 D))[\tau_{\xi_2D}(D), \tau_{\xi_2D}(D)]}{\|D\|_F},$$
which together with \eqref{upb} and Assumption \ref{assum_hess} completes the proof.
\end{proof}

\begin{remark}
We verify from \eqref{taylor3} and \eqref{taylor4} that Lemma \ref{distbound} holds true even if the total energy functional $E(U)$ is only of second order differentiable.
\end{remark}

\section{The Newton method} \label{sec-alg}
For any iteration point $U\in \mathcal{M}^N_{N_g}$, the central computation of a Newton algorithm is determining search direction $D\in\mathcal{T}_{[U]}\mathcal{G}^N_{N_g}$ such that
\begin{equation}\label{newton-equation}
\nabla^2_G E(U)[D]+\nabla_GE(U)=\bf{0}.
\end{equation}
In practice, it is usually impossible to get the exact solution of \eqref{newton-equation}. We may instead require \eqref{newton-equation} to be solved approximately
such that the search direction $D\in\mathcal{T}_{[U]}\mathcal{G}^N_{N_g}$ satisfies
\begin{equation}\label{sd}
\|\nabla^2_G E(U)[D]+\nabla_GE(U)\|_F\le\sigma\|\nabla_G E(U)\|_F,
\end{equation}
for some $\sigma\in(0,\frac{1}{2})$.

After a suitable direction $D$ is found, we note that $U + tD$ no longer belongs to Stiefel manifold $\mathcal{M}^N_{N_g}$ as long as $tD\ne{\bf 0}$. Therefore, some orthogonalization strategies, which are called ``retraction", are then required to be applied \cite{AbMaSe}.

\setcounter{footnote}{0}

For any given manifold $\mathcal{M}$, $U\in\mathcal{M}$ and operator $\mathcal{K}:\mathcal{T}_U\mathcal{M} \to \mathcal{M}$, we denote the derivative of $\mathcal{K}$ by $\textup{d}\mathcal{K}(\hat{D}):\mathcal{T}_{\hat{D}}\mathcal{T}_{U}\mathcal{M}\to\mathcal{T}_{U}\mathcal{M}$\footnote[1]{For any $\hat{D}\in\mathcal{T}_{U}\mathcal{M}$, the linear space $\mathcal{T}_{\hat{D}}\mathcal{T}_U\mathcal{M}$ is isomorphic to $\mathcal{T}_U\mathcal{M}$. Hence, $\textup{d}\mathcal{K}$ can be viewed as a mapping within $\mathcal{T}_{U}\mathcal{M}$.}, which satisfies
\begin{equation}\label{differential}
\lim_{\|\delta D\|\to 0}\frac{\|\mathcal{K}(\hat{D}+\delta D) - \mathcal{K}(\hat{D}) - \textup{d}\mathcal{K}(\hat{D})[\delta D]\|}{\|\delta D\|}=0.
\end{equation}
The ``retraction" is then defined as follows \cite{AbMaSe}.

\begin{definition}
A retraction $\mathcal{R}: \ \mathcal{T}\mathcal{M} \to \mathcal{M}$ on a manifold $\mathcal{M}$ is a smooth mapping satisfying
\begin{eqnarray*}
\mathcal{R}_{U}({\bf 0}) &=& U, \\
\textup{d}\mathcal{R}_{U}({\bf 0}) &=& \textup{Id}_{\mathcal{T}_{U}\mathcal{M}},
\end{eqnarray*}
where $\mathcal{R}_{U}$ is the restriction of $\mathcal{R}$ to $\mathcal{T}_{U}\mathcal{M}$ when $U\in\mathcal{M}$, $\bf 0$ denotes the zero element in $\mathcal{T}_{U}\mathcal{M}$, and $\textup{Id}_{\mathcal{T}_{U}\mathcal{M}}$ is the identity mapping on $\mathcal{T}_{U}\mathcal{M}$.
\end{definition}

For simplicity, we introduce a macro $\text{ortho}(U,D,t)$ to denote one step starting from point $U \in \mathcal{M}^N_{N_g}$ with search direction $D$ and step size $t$ to next  point, which is also in $\mathcal{M}^N_{N_g}$. More specifically, to be a retraction, operator $\text{ortho}(U,D,t)$ should satisfy
\begin{eqnarray}\label{retraction}
\text{ortho}(U,D,0)=U, \\
\frac{\partial}{\partial t}\textup{ortho}(U,D,0)=D. \notag
\end{eqnarray}
If \eqref{retraction} holds true for all $U\in\mathcal{M}^N$ and $D\in\mathcal{T}_{[U]}\mathcal{G}^N$, then the operator $\text{ortho}(U,D,t)$ is indeed a retraction \cite{AbMaSe}.

Another important issue in a Newton method is the step size. We notice that the initial step sizes for Newton methods are often chosen as unit step size, i.e., constant step size 1
(c.f., \cite{EAS, HMWY, Smith93, Smith94, ZBJ}) because \eqref{newton-equation} has indeed found a minimizer of second order Taylor expansion of the objective function with respect to the search direction under the constraint that the step size equal to 1.
Due to the fact that we do not require the Newton direction to be obtained exactly, we choose the initial step size at the $n$-th iteration to be
\begin{equation}\label{init_step}
t_n^{\textup{init}} = \frac{-\langle\nabla_G E(U_n), D_n\rangle}{\nabla_G^2 E(U_n)[D_n, D_n]},
\end{equation}
which is the so-called Hessian based step size in \cite{DLZZ}. Note that if $D_n$ solves \eqref{newton-equation}, then $t_n^{\textup{init}} = 1$ by its definition, which degenerates to the step size choice of classic Newton method. The monotone backtracking procedure is then applied to ensure the convergence of the algorithm.
We can then propose our Newton algorithm with macro $\textup{ortho}(U, D, t)$ as Algorithm \ref{Newton}.

\begin{algorithm}\label{Newton}
\caption{Newton method}
 Give $\epsilon,  q\in (0,1), \eta \in (0,\frac{1}{4})$,
  initial data $U_0, \ s.t. \  U_0^TU_0 = I_N$,
  calculate gradient $\nabla_G E(U_0)$, let $n = 0$\;
 \While{$\| \nabla_G E(U_n) \|_F> \epsilon$}{
  Choose a suitable $\sigma_n\in(0,1)$;

  Find $D_n\in\mathcal{T}_{[U_n]}\mathcal{G}^N_{N_g}$
  such that
  \begin{equation}\label{directionappro}
\|\nabla^2_G E(U_n)[D_n]+\nabla_GE(U_n)\|_F\le\sigma_n\|\nabla_G E(U_n)\|_F;
\end{equation}

Calculate the step size
     \begin{equation*}
       t_n = t_n^{\textup{init}}q^{m_n},
     \end{equation*}
     where $t_n^{\textup{init}}$ is defined in \eqref{init_step} and $m_n \in \mathbb{N}$ is the smallest nonnegative integer satisfying
     \begin{align}\label{armijo}
  E(\textup{ortho}(U_n,D_n,t_n)) \leq  \ E(U_n) + \eta t_n\langle\nabla_G E(U_n),  D_n\rangle;
\end{align}

  Update
      $U_{n+1}=\textup{ortho}(U_n,D_n,t_n)$;

  Let $n=n+1$, calculate gradient $\nabla_G E(U_n)$\;
 }
\end{algorithm}

\subsection{Convergence}
To ensure the convergence of Algorithm \ref{Newton}, we require that the retraction we use in Algorithm \ref{Newton} satisfies the following assumption which has been previously used in \cite{DLZZ, JD}.

\begin{assume}\label{diff}
There exists a constant $C_1 > 0$, such that
\begin{equation}\label{prop4.1}
\| \textup{ortho}(U, D, t) - U \|_F \le C_1t \|  D \|_F, \ \forall \  t\ge0,
\end{equation}
and
\begin{equation}\label{prop4.2}
\begin{split}
 \| \frac{\partial}{\partial t}\textup{ortho}(U,D,t) - D \|_F \le C_1t \| D\|_F^2, \ \forall \  t\ge0.
\end{split}
\end{equation}
\end{assume}

\begin{remark}
It has been proven in \cite{DLZZ} that Assumption \ref{diff} holds true for several retractions including the so-called QR, PD and WY strategies. We refer to Appendix \ref{app_retraction} for more details.
\end{remark}

\begin{remark}
Assumption \ref{diff} leads to
\begin{eqnarray*}
\|\textup{ortho}(U,D,t)-U-tD\|_F&=&t\| \frac{\partial}{\partial t}\textup{ortho}(U,D,\xi)-D\|_F \\ \notag
&\leq& C_1t\xi\|D\|_F^2\leq C_1t^2\|D\|_F^2.
\end{eqnarray*}
Thus, for any orthogonality preserving strategy $\textup{ortho}_1(U,D,t)$ and $\textup{ortho}_2(U,D,t)$ satisfying Assumption \ref{diff}, there holds
\begin{eqnarray}\label{differfromuptd}
&&\|\textup{ortho}_1(U,D,t)-\textup{ortho}_2(U,D,t)\|_F \notag \\
&\leq&\|\textup{ortho}_1(U,D,t)-U-tD\|_F + \|\textup{ortho}_2(U,D,t)-U-tD\|_F \notag \\
&\leq& 2C_1t^2\|D\|_F^2.
\end{eqnarray}
\end{remark}

Before starting to prove the convergence, we need the following estimation.
\begin{lemma}\label{stepsizebound}
Let Assumption \ref{assum_hess} hold true. If $[U] \in B([U^\ast],\delta_1)$ and $D\in\mathcal{T}_{[U]}\mathcal{G}_{N_g}^N$ satisfy \eqref{sd}, then
\begin{equation}\label{EequivD}
\frac{\nu_1}{1+\sigma}\|D\|_F\le\|\nabla_G E(U)\|_F\le\frac{C}{1-\sigma}\|D\|_F,
\end{equation}
and
$$|\frac{-\langle\nabla_G E(U), D\rangle}{\nabla_G^2 E(U)[D, D]}-1|\le \frac{C\sigma}{\nu_1(1-\sigma)}.$$
\end{lemma}
\begin{proof}
We see from \eqref{sd} that
\begin{equation*}
\sigma\|\nabla_G E(U)\|_F\ge\|\nabla^2_G E(U_n)[D_n]\|_F-\|\nabla_G E(U)\|_F,
\end{equation*}
which together with \eqref{pos-bnd} leads to
\begin{equation*}
\|\nabla_G E(U)\|_F\ge\frac{1}{1+\sigma}\|\nabla^2_G E(U)[D]\|_F\ge\frac{\|\nabla^2_G E(U)[D, D]\|_F}{(1+\sigma)\|D\|_F}\ge\frac{\nu_1}{1+\sigma}\|D\|_F.
\end{equation*}
Similarly, we have
$$(1-\sigma)\|\nabla_G E(U)\|_F\le\|\nabla^2_G E(U)[D]|_F\le C\|D\|_F,$$
i.e.,
$$\|\nabla_G E(U)\|_F\le\frac{C}{1-\sigma}\|D\|_F.$$
In addition, we obtain from \eqref{sd} that
$$|\langle\nabla_G E(U), D\rangle+\nabla_G^2 E(U)[D, D]|\le\sigma\|\nabla_G E(U)\|_F\|D\|_F,$$
or equivalently,
$$|\frac{-\langle\nabla_G E(U), D\rangle}{\nabla_G^2 E(U)[D, D]}-1|\le\frac{\sigma\|\nabla_G E(U)\|_F\|D\|_F}{\nabla_G^2 E(U)[D, D]}.$$
By using 
\eqref{pos-bnd} and \eqref{EequivD}, we get that
$$\frac{\|\nabla_G E(U)\|_F\|D\|_F}{\nabla_G^2 E(U)[D, D]}\le\frac{\|\nabla_G E(U)\|_F}{\nu_1\|D\|_F}\le\frac{C}{\nu_1(1-\sigma)}$$
and hence complete the proof.
\end{proof}

Now we turn to show our theory. To use Assumption \ref{assum_hess} in our analysis, we first provide a sufficient condition for keeping every iteration point in $B([U^\ast], \delta_1)$ as long as initial guess $U_0$ is close enough to $U^\ast$.

We see from Assumption \ref{assum_hess} that, for any $\delta_2 \in (0, \delta_1/(1+\frac{3C_1}{\nu_1}L_1))$,  there exists an $E_0>E(U^\ast)$ and the corresponding level set
\begin{equation}\label{level-set-L}
\mathcal{L} = \{[U]\in \mathcal{G}_{N_g}^N: E(U)\le E_0\},
\end{equation}
 such that
\begin{equation}\label{level}
\{[U]: [U] \in \mathcal{L} \cap B([U^\ast], \delta_1)\}
\subset B([U^\ast], \delta_2).
\end{equation}
We have the following lemma.

\begin{lemma} \label{lema1}
Let Assumptions \ref{assum_lip} and \ref{assum_hess} hold true. If $[U_0] \in B([U^\ast], \delta_2) \cap \mathcal{L}$, then there exists a sequence $\{\sigma_n\}_{n\in \mathbb{N}_0}\subset(0,1)$
such that for the sequence $\{U_n\}_{n\in \mathbb{N}_0}$ generated by
Algorithm \ref{Newton}, there holds $$[U_n] \in B([U^\ast], \delta_2) \cap \mathcal{L} \subset B([U^\ast], \delta_1), \forall \ n \in \mathbb{N}_{0}.$$
\end{lemma}
\begin{proof}
Let us prove the conclusion by induction. Since $[U_0] \in
B([U^\ast], \delta_2)\cap \mathcal{L}$,
we see that the conclusion is true for $n=0$. We assume that $[U_n] \in
B([U^\ast], \delta_2) \cap\mathcal{L}$, which implies
$[U_n] \in B([U^\ast], \delta_1)$. Lemma \ref{stepsizebound} then gives that
\begin{equation}\label{DleE}
\|D_n\|\le\frac{1+\sigma_n}{\nu_1}\|\nabla_G E(U_n)\|,
\end{equation}
and $t_n\le t_n^{\textup{init}}\le2$ as long as $\sigma_n\le\frac{\nu_1}{\nu_1+C}.$
Hence, we obtain from Assumption \ref{diff} that
\begin{align*}
\| U_{n+1} - U_n\|_F  & \le C_1 t_n\| D_n\|_F  \le 2C_1 \| D_n\|_F  \\
                      & \le \frac{2C_1(1+\sigma_n)}{\nu_1}\| \nabla_G E(U_n)\|_F.
\end{align*}
By the definition of $\text{dist}_{F}([U_n], [U^\ast])$, there exists  $P_n \in \mathcal{O}^{N}$,
such that
\begin{eqnarray*}
 \text{dist}_{F}([U_n], [U^\ast]) = \| U_n - U^\ast P_n \|_F,
\end{eqnarray*}
which together with $\nabla_G E(U^\ast) = 0$ and Assumption
\ref{assum_lip} leads to
 \begin{eqnarray*}
  \| U_{n+1} - U_n\|_F &\le& \frac{2C_1(1+\sigma_n)}{\nu_1}\| \nabla_G E(U_n)- \nabla_G E(U^\ast)P_n\|_F  \\
    &=& \frac{2C_1(1+\sigma_n)}{\nu_1}\| \nabla_G E(U_n)- \nabla_G E(U^\ast P_n)\|_F \\
    &\le& \frac{2C_1(1+\sigma_n)}{\nu_1}L_1\| U_n - U^\ast P_n\|_F  \le \frac{2C_1(1+\sigma_n)}{\nu_1}L_1\delta_2,
\end{eqnarray*}
where the assumption that $[U_n] \in B([U^\ast], \delta_2) \cap \mathcal{L}$ is used in the last inequality.
Consequently,
\begin{eqnarray*}
\text{dist}_{F}([U_{n+1}], [U^\ast]) &\leq & \| U_{n+1}-U^\ast P_n\|_F  \\
 &\leq & \| U_{n+1}-U_n\|_F+\| U_n-U^\ast P_n\|_F \\
  &\le& \| U_{n+1} - U_n\|_F + \delta_2 \le (1+\frac{2(1+\sigma_n)C_1}{\nu_1}L_1)\delta_2\le \delta_1,
\end{eqnarray*}
which means that $[U_{n+1}] \in B([U^\ast], \delta_1)$.

Besides, We obtain from \eqref{directionappro} and the triangular inequality that
\begin{equation}\label{resi_bound1}
|\langle\nabla_G E(U_n),D_n\rangle+\nabla^2_G E(U_n)[D_n, D_n]|\le\sigma_n\|\nabla_G E(U_n)\|_F\|D_n\|_F,
\end{equation}
and from Lemma \ref{stepsizebound} that
\begin{equation}\label{EleD}
\|\nabla_G E(U_n)\|_F\le\frac{C}{1-\sigma_n}\|D_n\|_F.
\end{equation}
Inserting \eqref{EleD} into \eqref{resi_bound1}, we get
\begin{equation}\label{absbound}
|\langle\nabla_G E(U_n),D_n\rangle+\nabla^2_G E(U_n)[D_n, D_n]|\le\frac{\sigma_n}{1-\sigma_n}C\|D_n\|_F^2,
\end{equation}
which indicates that for any $\sigma_n\le\frac{\nu_1}{\nu_1+C}$,
\begin{eqnarray*}
\langle\nabla_G E(U_n),D_n\rangle&\le&\frac{\sigma_n}{1-\sigma_n}C\|D_n\|_F^2-\nabla^2_G E(U_n)[D_n, D_n] \\&\le&(\frac{\sigma_nC}{1-\sigma_n}-\nu_1)\|D_n\|_F^2\le0.
\end{eqnarray*}
Noticing that \eqref{armijo} gives
\begin{align*}
  E(U_{n+1}) \leq  \ E(U_n) + \eta t_n\langle\nabla_G E(U_n),D_n\rangle  \leq E(U_n),
\end{align*}
we have $[U_{n+1}] \in B([U^\ast], \delta_1) \cap \mathcal{L}\subset B([U^\ast], \delta_2)$ where \eqref{level} is used. Finally, we obtain $[U_{n+1}] \in B([U^\ast], \delta_1) \cap \mathcal{L}\subset B([U^\ast], \delta_2)\cap \mathcal{L}$ and complete the proof.
\end{proof}

We are now able to show our main result which can be proved by the similar approach as that in Theorem 2.7 of \cite{DZZ}.

\begin{theorem}\label{thm-conv}
Let Assumptions \ref{assum_lip} and \ref{assum_hess} hold true. If retraction $\textup{ortho}(U, D, t)$ is chosen to satisfy Assumption \ref{diff}, $[U_0] \in B([U^\ast], \delta_2) \cap \mathcal{L}$, then there exists a sequence $\{\sigma_n\}_{n\in \mathbb{N}_0}\subset(0,1)$ such that for the sequence $\{U_n\}_{n\in \mathbb{N}_0}$  generated by Algorithm \ref{Newton}, there holds either $\|\nabla_G E(U_n)\|_F = 0$ for some positive integer $n$ or
\begin{equation}\label{eq-conv}
\lim_{n\to\infty} \| \nabla_G E(U_n)\|_F  = 0.
\end{equation}
\end{theorem}
\begin{proof}
Note that Lemma \ref{lema1} gives that $[U_n]\in B([U^\ast], \delta_1), n\in\mathbb{N}_0$  as long as $\sigma_n\le\frac{\nu_1}{\nu_1+C}~(n=0,1,2,\cdots)$. Hence \eqref{pos-bnd} holds true during the iterations. If $\|\nabla_G E(U_n)\|_F = 0$ for some positive integer $n$, the conclusion is trivial. We assume otherwise.

Let $\sigma_n\le\frac{\nu_1}{\nu_1+2C}~(n=0,1,2,\cdots)$, then we have from \eqref{directionappro}, \eqref{EleD} and \eqref{pos-bnd} that
\begin{eqnarray}\label{negative_direction}
\langle\nabla_G E(U_n),D_n\rangle&\le&\sigma_n\|\nabla_G E(U_n)\|_F\|D_n\|_F-\nabla^2_G E(U_n)[D_n, D_n] \\
&\le&\frac{\sigma_n}{1-\sigma_n}C\|D_n\|_F^2-\nabla^2_G E(U_n)[D_n, D_n] \notag \\
&\le&\big(\frac{\sigma_n}{1-\sigma_n}C-\nu_1\big)\|D_n\|_F^2 \notag\\
&\le&-\frac{\nu_1}{2}\|D_n\|_F^2<0.\notag
\end{eqnarray}
Furthermore, we see that
\begin{eqnarray}\label{not_ortho}
\frac{-\langle\nabla_G E(U_n),D_n\rangle}{\|\nabla_G E(U_n)\|_F^2}&\ge&\frac{(1-\sigma_n)^2}{C^2}\frac{-\langle\nabla_G E(U_n),D_n\rangle}{\|D_n\|_F^2} \\
&\ge& \frac{(1-\sigma_n)^2\nu_1}{2C^2}\ge\frac{2\nu_1}{(\nu_1+2C)^2} \notag
\end{eqnarray}
where \eqref{EleD} is used again. Combining \eqref{negative_direction}, \eqref{not_ortho}, \eqref{DleE} and noting that $$\|\nabla_G E(U_n)\|_F\le\|\nabla E(U_n)\|_F\le C_0, n=0,1,\cdots$$ we complete the proof by using Theorem 2.7 of \cite{DZZ}.
%
%
\end{proof}

\begin{remark}
Let Assumptions \ref{assum_lip} and \ref{assum_hess} hold true, the retraction $\textup{ortho}(U, D, t)$ is chosen to satisfy Assumption \ref{diff}, $[U_0] \in B([U^\ast], \delta_2) \cap \mathcal{L}$ and $\sigma_n\le\frac{\nu_1}{\nu_1+2C}, n\in\mathbb{N}_0$. Then we see from Lemma 2.6 and the proof of Theorem 4.7 of \cite{DLZZ} that there exists an unique local minimizer in $B([U^\ast], \delta_1)$ which is $[U^\ast]$ itself owing to Assumption \ref{assum_hess} and
\begin{equation}\label{dist=0}
\lim_{n\to\infty} \textup{dist}_{F}([U_n], [U^\ast]) = 0.
\end{equation}
\end{remark}

\subsection{Convergence rate}
We are going to show the convergence rate of Algorithm \ref{Newton} assuming that the retraction used in Algorithm \ref{Newton} satisfies the following estimation.
\begin{assume}\label{lema-smdif}
Retraction $\textup{ortho}(U,D,t)$ satisfies that for any $D\in\mathcal{T}_{[U]}\mathcal{G}^N_{N_g}$, and $t\in[0,\frac{1}{\|D\|_F})$, there holds
\begin{equation}\label{smalldiff}
E(\textup{ortho}(U,D,t))-E(\exp_{[U]}(tD))= O(t^3\|D\|_F^3).
\end{equation}
\end{assume}

We note that all the well known retractions for orthogonality constrained optimization method satisfy Assumption \ref{lema-smdif}. We refer to Appendix \ref{app_retraction} for more detailed discussions. We then present that the backtracking is not required for any large enough $n$.

\begin{lemma}\label{noback}
Let Assumptions \ref{assum_lip} and \ref{assum_hess} hold true. Suppose retraction $\textup{ortho}(U, D, t)$ satisfies Assumption \ref{lema-smdif}, and the sequence $\{U_n\}_{n\in \mathbb{N}_0}, \{t_n\}_{n\in \mathbb{N}_0}$  are generated by Algorithm \ref{Newton}. If $[U_0] \in B([U^\ast], \delta_2) \cap \mathcal{L}$, then there exists a sequence $\{\sigma_n\}_{n\in \mathbb{N}_0}\subset(0,1)$ such that $\lim\limits_{n\to\infty} m_n\to 0$ in Algorithm \ref{Newton}.
\end{lemma}
\begin{proof}
We see from \eqref{taylor1} and \eqref{smalldiff} that
\begin{eqnarray*}
&&E(\textup{ortho}(U_n,D_n,t_n^{\textup{init}}))-E(U_n) \\
&=&E(\exp_{[U_n]}(t_n^{\textup{init}}D_n))-E(U_n) +E(\textup{ortho}(U_n,D_n,t_n^{\textup{init}}))-E(\exp_{[U_n]}(t_n^{\textup{init}}D_n)) \\
&=&t_n^{\textup{init}}\langle\nabla_GE(U_n),D_n\rangle+\frac{{t_n^{\textup{init}}}^2}{2}\nabla^2_GE(U_n)[D_n, D_n]+O(\|D_n\|_F^3).
\end{eqnarray*}
A simple calculations gives that $t_n^{\textup{init}}$ minimizes
$$t\langle\nabla_GE(U_n),D_n\rangle+\frac{t}{2}\nabla^2_GE(U_n)[D_n, D_n]$$
with respect to $t\in\mathbb{R}.$ Hence,
\begin{eqnarray}\label{differ}
&&E(\textup{ortho}(U_n,D_n,t_n^{\textup{init}}))-E(U_n) \\
&\le&\langle\nabla_GE(U_n),D_n\rangle+\frac{1}{2}\nabla^2_GE(U_n)[D_n, D_n]+O(\|D_n\|_F^3) \notag\\
&\le&\frac{1}{2}\langle\nabla_GE(U_n),D_n\rangle + \frac{1}{2}|\langle\nabla_GE(U_n),D_n\rangle+\nabla^2_GE(U_n)[D_n, D_n]| \notag\\
&&+ O(\|D_n\|_F^3) \notag\\
&\le&\frac{1}{2}\langle\nabla_GE(U_n),D_n\rangle + \frac{1}{2}\sigma_n\|\nabla_G E(U_n)\|_F\|D_n\|_F + O(\|D_n\|_F^3) \notag
\end{eqnarray}
where \eqref{resi_bound1} is used in the last inequality. Let $\sigma_n\le\frac{\nu_1}{\nu_1+2C}~(n=0,1,2,\cdots)$,
we obtain from \eqref{DleE} and \eqref{negative_direction} that
\begin{equation}\label{inversebound}
\|\nabla_G E(U_n)\|_F\|D_n\|_F\le-\frac{2C}{(1-\sigma_n)\nu_1}\langle\nabla_GE(U_n),D_n\rangle.
\end{equation}
In addition, we have $-\langle\nabla_GE(U_n),D_n\rangle\cong\|\nabla_G E(U_n)\|_F\|D_n\|_F$ since $$-\langle\nabla_GE(U_n),D_n\rangle\le\|\nabla_G E(U_n)\|_F\|D_n\|_F.$$
We then immediately see from Lemma \ref{stepsizebound} that $-\langle\nabla_GE(U_n),D_n\rangle\cong\|D\|_F^2$. Here, $A\cong B$ means that there exist some constants $0<\underline{c}\le\bar{c}<\infty$ such that $\underline{c}A\le B\le\bar{c}A$.

Now let $\sigma_n \le \frac{\nu_1}{\nu_1+4C}$, we obtain
\begin{eqnarray*}
0\le\frac{1}{2}\sigma_n\|\nabla_G E(U_n)\|_F\|D_n\|_F &\le& -\frac{1}{4}\langle\nabla_GE(U_n),D_n\rangle.
\end{eqnarray*}
As a result, \eqref{EleD} indicates that
\begin{equation}\label{differ2}
E(\textup{ortho}(U_n,D_n,t_n^{\textup{init}}))-E(U_n) \le\frac{1}{4}\langle\nabla_GE(U_n),D_n\rangle\big(1 + O(\|D_n\|_F)\big).
\end{equation}
Note that $\displaystyle\lim_{n\to\infty}\|D_n\|_F = 0$, \eqref{differ2} indicates that $t_n=t_n^{\textup{init}}$ satisfies the Armijo condition for sufficiently large $n$ with $\eta \in(0,\frac{1}{4})$, which completes the proof.
\end{proof}

Let $[U^\ast]$ be the unique local minimizer in $B([U^\ast], \delta_1)$, we have the following theorem.

\begin{theorem}\label{conv-rate}
Let Assumptions \ref{assum_lip} and \ref{assum_hess} hold true. Suppose retraction $\textup{ortho}(U, D, t)$ satisfies Assumption \ref{lema-smdif}, and the sequence $\{U_n\}_{n\in \mathbb{N}_0}$  is generated by Algorithm \ref{Newton}. If $[U_0] \in B([U^\ast], \delta_2) \cap \mathcal{L}$, then there exists a sequence $\{\sigma_n\}_{n\in \mathbb{N}_0}\subset(0,1)$ such that $\{U_n\}_{n\in \mathbb{N}_0}$ converges to $[U^\ast]$ at least quadratically for sufficiently large $n$, namely,
\begin{equation}
\textup{dist}_{geo}([U_{n+1}],[U^\ast])\leq\zeta {\textup{dist}_{geo}([U_n],[U^\ast])}^2
\end{equation}
for some constant $\zeta>0$.
\end{theorem}
\begin{proof}
For simplicity, we denote $d_n=\textup{dist}_{geo}([U_n],[U^\ast])$. Lemma \ref{noback} implies that $U_{n+1}=\textup{ortho}(U_n, D_n, 1)$ for sufficiently large $n$. Then we have from Lemma \ref{distbound} and \eqref{absbound} that
\begin{eqnarray*}
\frac{\nu_1}{2}{d_{n+1}}^2&\le& E(U_{n+1})-E(U^\ast) \\
&=&E(U_{n+1})-E(U_n)+E(U_n)-E(U^\ast) \\
&\leq&\eta\langle\nabla_GE(U_n),D_n\rangle+\frac{C}{2}{d_n}^2  \\
&\leq&-\eta\nabla^2_GE(U_n)[D_n, D_n]+\frac{C\sigma_n}{1-\sigma_n}\|D_n\|_F^2+\frac{C}{2}{d_n}^2  \\
&\leq&(\frac{C\sigma_n}{1-\sigma_n}-\eta\nu_1)\|D_n\|_F^2+\frac{C}{2}{d_n}^2 \\
&\leq&-\frac{\eta\nu_1}{2}\|D_n\|_F^2+\frac{C}{2}{d_n}^2,
\end{eqnarray*}
as long as $\sigma_n\le\frac{\eta\nu_1}{\eta\nu_1+2C}$.

Furthermore, \eqref{taylor2} indicates that there exists a $\xi\in(0,t_n)$, such that
\begin{eqnarray*}
&&\|\nabla_G E(\exp_{[U_n]}(t_nD_n)\|_F^2 \\
&=&\langle\tau_{t_nD_n}^{-1}\big(\nabla_G E(\exp_{[U_n]}(t_nD_n))\big),\tau_{t_nD_n}^{-1}\big(\nabla_G E(\exp_{[U_n]}(t_nD_n))\big)\rangle \\
&=&\langle\nabla_G E(U_n)+t_n\nabla_G^2E(U_n)[D_n],\tau_{t_nD_n}^{-1}\big(\nabla_G E(\exp_{[U_n]}(t_nD_n))\big)\rangle \\
&&+\frac{t_n^2}{2}\nabla_G^3 E(\exp_{[U_n]}(\xi D_n))[\tau_{\xi D_n} (D_n), \tau_{\xi D_n} (D_n), \tau_{\xi D_n} \big(\tau_{t_nD_n}^{-1}(\nabla_G E(\exp_{[U_n]}(D_n)))\big)].
\end{eqnarray*}
Applying \eqref{directionappro}, Lemma \ref{stepsizebound} and \eqref{3upb}, we have
\begin{eqnarray*}
&&\|\nabla_G E(\exp_{[U_n]}(t_nD_n)\|_F^2 \\
&=&\langle\nabla_G E(U_n)+\nabla_G^2E(U_n)[D_n],\tau_{t_nD_n}^{-1}\big(\nabla_G E(\exp_{[U_n]}(t_nD_n))\big)\rangle \\
&&+\langle(t_n-1)\nabla_G^2E(U_n)[D_n],\tau_{t_nD_n}^{-1}\big(\nabla_G E(\exp_{[U_n]}(t_nD_n))\big)\rangle \\
&&+\frac{t_n^2}{2}\nabla_G^3 E(\exp_{[U_n]}(\xi D_n))[\tau_{\xi D_n} (D_n), \tau_{\xi D_N} (D_n), \tau_{\xi D_n} \big(\tau_{t_nD_n}^{-1}(\nabla_G E(\exp_{[U_n]}(t_nD_n)))\big)] \\
&\leq&\big(\sigma_n\|\nabla_GE(U_n)\|_F+\frac{C^2\sigma_n}{\nu_1(1-\sigma_n)}\|D_n\|_F+\frac{Ct_n^2}{2}\|D_n\|_F^2\big)\|\nabla_G E(\exp_{[U_n]}(t_nD_n)\|_F.
\end{eqnarray*}
Namely,
\begin{equation*}
\|\nabla_G E(\exp_{[U_n]}(D_n))\|_F\leq\sigma_n\|\nabla_GE(U_n)\|_F + \frac{C^2\sigma_n}{\nu_1(1-\sigma_n)}\|D_n\|_F + \frac{Ct_n^2}{2}\|D_n\|_F^2.
\end{equation*}
Now choose $\sigma_n \le \min\{\eta\nu_1/(\eta\nu_1+2C), \|\nabla_G E(U_n)\|_F\}$, we have $t_n\le\frac{3}{2}$ and then
\begin{equation}\label{GLD2}
\|\nabla_G E(\exp_{[U_n]}(D_n))\|_F\leq\frac{2(1+\frac{C}{\nu_1})(2C+\eta\nu_1)^2+9C}{8}\|D_n\|_F^2,
\end{equation}
where \eqref{EleD} is used.
Combining \eqref{differfromuptd} and \eqref{GLD2}, we obtain
\begin{eqnarray*}
&&\|\nabla_G E(U_{n+1})\|_F \\
&\leq&\|\nabla_G E(U_{n+1})-\nabla_G E(\exp_{[U_n]}(t_nD_n)\|_F+\|\nabla_G E(\exp_{[U_n]}(t_nD_n)\|_F \\
&\leq&L_1\|\textup{ortho}(U_n,D_n,t_n)-\exp_{[U_n]}(t_nD_n)\|_F+\frac{2(1+\frac{C}{\nu_1})(2C+\eta\nu_1)^2+9C}{8}\|D_n\|_F^2 \\
&\leq&\frac{24L_1C_2+9C+2(1+\frac{C}{\nu_1})(2C+\eta\nu_1)^2}{4}\|D_n\|_F^2.
\end{eqnarray*}
Consequently,
\begin{eqnarray*}
\frac{C}{2}{d_n}^2&\geq&\frac{\nu_1}{2}{d_{n+1}}^2+\frac{\eta\nu_1}{2}\|D_n\|_F^2 \\
&\geq&\frac{2\eta\nu_1}{24L_1C_2+9C+2(1+\frac{C}{\nu_1})(2C+\eta\nu_1)^2}\|\nabla_G E(U_{n+1})\|_F \\
&\geq&\frac{2\eta\nu_1^2}{24L_1C_2+9C+2(1+\frac{C}{\nu_1})(2C+\eta\nu_1)^2}d_{n+1},
\end{eqnarray*}
where Lemma \ref{distbound} is used in the last line.
Finally, we get that
\begin{equation*}
d_{n+1}\leq\zeta {d_n}^2,
\end{equation*}
and $\zeta$ can be chosen as $\frac{C(24L_1C_2+9C+2(1+\frac{C}{\nu_1})(2C+\eta\nu_1)^2)}{4\eta\nu_1^2}$.
\end{proof}

\begin{remark}
If energy functional $E(U)$ is of second order differentiable only, then our theoretical results still hold true under the assumption that the Grassmann Hessian of $E(U)$ is locally Lipschitz continuous in the following sense: there exists a constant $L>0$ such that for all $[U], [V]\in B([U^\ast],\delta_1)$, 
\begin{equation}\label{Hess-Lip1}
\|\tau_{D}^{-1}\big(\nabla_G^2 E\big(V\big)[\tau_{D}(\cdot)]\big)-\nabla_G^2 E\big(U\big)[\cdot]\|\le L\textup{dist}_{geo}([U],[V]),
\end{equation}
where $D\in\mathcal{T}_{[U]}\mathcal{G}_{N_g}^N$ such $[\exp_{[U]}(D)]=[V]$ (c.f. \eqref{geodesic1}).
Equivalently, Lipschitz condition \eqref{Hess-Lip1} can be rewritten as
\begin{equation}\label{Hess-Lip}
\|\tau_{D}^{-1}\big(\nabla_G^2 E\big(\exp_{[U]}(D)\big)[\tau_{D}(\cdot)]\big)-\nabla_G^2 E\big(U\big)[\cdot]\|\le L\|D\|_F.
\end{equation}
In fact, we get from \eqref{Hess-Lip} that for $D,\tilde{D},\tilde{\tilde{D}}\in\mathcal{T}_{[U]}\mathcal{G}_{N_g}^N$
\begin{eqnarray*}
&&|\nabla_G^2 E(\exp_{[U]}(D))[\tau_{D}(\tilde{D}),\tau_{D}(\tilde{\tilde{D}})]-\nabla_G^2 E(U)[\tilde{D},\tilde{\tilde{D}}]| \\
&=&|\langle\tau_{D}^{-1}\big(\nabla_G^2 E(\exp_{[U]}(D))[\tau_D(\tilde{D})]\big) - \nabla_G^2 E(U)[\tilde{D}], \tilde{\tilde{D}}\rangle| \\
&\le&\| \tau_{D}^{-1}\big(\nabla_G^2 E(\exp_{[U]}(D))[\tau_D(\tilde{D})]\big) - \nabla_G^2 E(U)[\tilde{D}]\|_F \|\tilde{\tilde{D}}\|_F \\
&\le&\|{\tau_D}^{-1}\big(\nabla_G^2 E\big(\exp_{[U]}(D)\big)[\tau_D(\cdot)]\big)-\nabla_G^2 E(U)[\cdot]\|\|\tilde{D}\|_F\|\tilde{\tilde{D}}\|_F \\
&\le&L\|D\|_F\|\tilde{D}\|_F\|\tilde{\tilde{D}}\|_F.
\end{eqnarray*}
Thus, we obtain from \eqref{taylor3} and \eqref{taylor4} that if $t\ge0$ is bounded, then
\begin{eqnarray}\label{approx1}
&&|E(\exp_{[U]}(tD))-E(U) - t\langle\nabla_GE(U),D\rangle -  \frac{t^2}{2}\nabla_G E(U)[D,D]| \\
&\le&t\int_0^t (1-\frac{s}{t})|\big(\nabla_G^2 E(\exp_{[U]}(sD))[\tau_{sD}(D),\tau_{sD}(D)]-\nabla_G E(U)[D,D]\big)|ds \notag\\
&\le&Lt\int_0^t (1-\frac{s}{t})sds\|D\|_F^3 = \frac{Lt^3}{6}\|D\|_F^3 = O(\|D\|_F^3).\notag
\end{eqnarray}
Similarly, we have
\small\begin{eqnarray}\label{approx2}
&&|\|\nabla_G E(\exp_{[U]}(tD)\|_F^2 - \langle\nabla_G E(U)+t\nabla_G^2E(U)[D],{\tau_{tD}}^{-1}\big(\nabla_G E(\exp_{[U]}(tD))\big)\rangle|  \\
&=&|\int_0^t \nabla^2_GE(\exp_{[U]}(sD))[\tau_{sD}(D), \tau_{sD}\big(\tau_{tD}^{-1}(\nabla_G E(\exp_{[U_n]}(t_nD_n)))\big)]ds \notag\\
&&-\int_0^t \nabla_G^2 E(U)[D,\tau_{tD}^{-1}\big(\nabla_G E(\exp_{[U]}(tD))\big)]ds| \notag\\
&\le& \int_0^t \langle\tau_{sD}^{-1}\big(\nabla_G^2 E(\exp_{[U]}(sD))[\tau_{sD}(D)]\big)-\nabla_G^2 E(U)[D], \tau_{tD}^{-1}\big(\nabla_G E(\exp_{[U]}(tD))\big)\rangle ds \notag\\
&\le&\frac{Lt^2}{2}\|\nabla_G E(\exp_{[U]}(tD)\|_F\|D\|_F^2.\notag
\end{eqnarray}
Note that  \eqref{approx1} and \eqref{approx2} valid Lemma \ref{noback} and Theorem \ref{conv-rate}, respectively.
\end{remark}

\section{Implementation issues} \label{sec-spbs}
In this section, we address how to choose suitable search directions and investigate an adaptive step size strategy \cite{DZZ} to make our algorithm more practical.
\subsection{Search direction solver}
It is a key issue to determine the search direction $D_n$ efficiently in Algorithm \ref{Newton}. There are some existing works finding the direction by approximately solving
\begin{equation}\label{newton-equation-n}
\nabla^2_G E(U_n)[D]+\nabla_GE(U_n) = 0
\end{equation}
as some linear systems \cite{HMWY, ZBJ}. Instead of solving any linear systems, we present a new perspective to obtain desired search direction $D_n\in\mathcal{T}_{[U_n]}\mathcal{G}_{N_g}^N$ which satisfies \eqref{directionappro}.

We observe that \eqref{newton-equation-n} is the first order necessary condition of the following minimization problem:
\begin{equation}\label{miniprob}
\min_{D\in \mathcal{D}_{[U_n]}} \ \ \ \langle\nabla_G E(U_n), D\rangle+ \frac{1}{2}\nabla^2_G E(U_n)[D, D].
\end{equation}
The solution of \eqref{miniprob} is also the solution of \eqref{newton-equation-n}. Therefore, we turn to solve \eqref{miniprob} instead of solving \eqref{newton-equation-n} directly.

In addition, we see that solving \eqref{miniprob} is equivalent to solving an orthogonality constrained problem which is a consequence of the following lemma.
\begin{lemma}\label{fullmap}
For any $D\in \mathcal{D}_{[U_n]}$, there exist $U\in \mathcal{M}^N_{N_g}$ such that
\begin{equation}\label{equiv}
(I_N - U_nU_n^T)U=D.
\end{equation}
\end{lemma}
\begin{proof}
For any $D\in \mathcal{D}_{[U_n]}$, suppose $D=ASB^T$ to be the SVD of $D$ with $A\in\mathcal{M}_{N_g}^N$. Let $$\tilde{D}=A \arcsin{S}B^T.$$
We claim that $\tilde{D}\in\mathcal{D}_{[U_n]}$.

In fact, we see from $U_n^TD=\bf{0}$ that $$(U_n^TA)S=\bf{0}.$$ Suppose $(U_n^TA)=(z_{ij})_{i,j=1}^N$ and $S=\textup{Diag}\big((s_j)_{j=1}^N\big)$, then $z_{ij}s_j=0$, which implies $z_{ij}=0$ or $s_j=0$.

Hence $U_n^T\tilde{D}=(U_n^TA)\arcsin{S}B^T$. Consider $U_n^T\tilde{D}B=(U_n^TA)\arcsin{S}$, whose $ij$-th element is $z_{ij}\arcsin{s_j}$. Note that $$z_{ij}\arcsin{s_j}=0$$ for either $z_{ij}=0$ or $s_j=0$. As a result, $$U_n^T\tilde{D}=U_n^T\tilde{D}BB^T=\bf{0}.$$
Now, let $U=\exp_{[U_n]}(\tilde{D})$, where $\exp_{[U_n]}(D)$ means the geodesic on Grassmann manifold $\mathcal{G}_{N_g}^N$ with starting point $U_n$ and along direction $D$. Then we see from Theorem 2.3 of \cite{EAS} that
\begin{eqnarray*}
U&=&\big(U_nB, A\big)\big(\cos(\arcsin{S}), S\big)^TB^T \\
&=&U_nB\cos(\arcsin{S})B^T+D,
\end{eqnarray*}
which leads to $(I_N - U_nU_n^T)U=D.$ In addition, we have $\|\tilde{D}\|_F = \textup{dist}_{geo}([U_n],[\exp_{[U_n]}(\tilde{D})])$ and complete the proof.
\end{proof}

Thanks to Lemma \ref{fullmap}, we turn to investigate
\begin{equation}\label{miniprob2}
\min_{U\in \mathcal{M}^N_{N_g}} \ \ \ \bar{E}_n(U),
\end{equation}
instead of solving \eqref{newton-equation-n}. Here, $$\bar{E}_n(U)=\langle\nabla_G E(U_n), (I_N - U_nU_n^T)U\rangle+ \frac{1}{2}\nabla^2_G E(U_n)[(I_N - U_nU_n^T)U, (I_N - U_nU_n^T)U].$$

We see from \cite{EAS} that the gradient of $\bar{E}_n(U)$ on $\mathcal{M}^N_{N_g}$ is
\begin{equation}\label{Ebar-gra}
\nabla_S \bar{E}_n(U) = \nabla \bar{E}_n(U)-U\nabla \bar{E}_n(U)^TU,
\end{equation}
where
$$\nabla \bar{E}_n(U) = \nabla_G E(U_n)+ \nabla_G^2 E(U_n)[(I_N-U_nU_n^T)U]$$
and the Hessian of $\bar{E}_n(U)$ on $\mathcal{M}^N_{N_g}$ is
\begin{eqnarray}\label{Ebar-hess}
\nabla^2_S \bar{E}_n(U)[\delta U_1, \delta U_2] &=& \nabla_G^2 E(U_n)[(I_N-U_nU_n^T)\delta U_1, \delta U_2] \notag\\
&&+\frac{1}{2} \textup{tr}\big((\nabla\bar{E}(U_n)^T\delta U_1U^T+U^T\delta U_1\nabla\bar{E}(U_n)^T)\delta U_2\big)  \notag \\
&&-\frac{1}{2} \textup{tr}\big((U^T\nabla\bar{E}(U_n)+\nabla\bar{E}(U_n)^TU)\delta U_1^T(I-UU^T)\delta U_2\big),
\end{eqnarray}
with $$\delta U_1, \delta U_2\in\mathcal{T}_U\mathcal{M}_{N_g}^N = \{W\in\mathbb{R}^{N_g\times N}|W^TU+U^TW = 0\}.$$

Suppose $\bar{U}$ is a solution of \eqref{miniprob2}, then $D_n=(I_N - U_nU_n^T)\bar{U}$ is a solution of \eqref{miniprob} which can be chosen as the Newton search direction. We see that \eqref{miniprob2} is an orthogonality constrained minimization problem and can be solved by an orthogonality constrained CG method, whose details are shown in Algorithm \ref{CG}.

\begin{algorithm}\label{CG}
\caption{Conjugate gradient method for solving \eqref{miniprob2}}
 Give $\gamma_1,\gamma_2, q\in(0,1)$, the initial data $U^{(0)}\in\mathcal{M}^{N}_{N_g}$, $\delta U^{(0)} = -\nabla_S\bar{E}_n(U^{(0)})$, set $k = 0$;

\While{not converge}{
\If{$\langle\delta U^{(k)},\nabla_S\bar{E}(U^{(k)})\rangle>0$}{
$$\delta U^{(k)} = -\delta U^{(k)};$$}
\If{$\frac{-\langle\delta U^{(k)},\nabla_S\bar{E}(U^{(k)})\rangle}{\|\nabla_S\bar{E}(U^{(k)})\|_F^2}<\gamma_1$}{
$$\delta U^{(k)} = -\nabla_S\bar{E}(U^{(k)});$$}

Calculate $$\alpha^{(k)}=\frac{-\langle\delta U^{(k)},\nabla_S\bar{E}(U^{(k)})\rangle}{\nabla_S^2 \bar{E}(U_n)[(I_N-U_nU_n^T)\delta U^{(k)}, \delta U^{(k)}]};$$

Update $$U^{(k+1)} = \textup{ortho}(U^{(k)}, \delta U^{(k)}, \alpha^{(k)});$$

\If{$\bar{E}(U^{(k+1)})-\bar{E}(U^{(k)})\ge\gamma_2\alpha^{(k)}\langle\delta U^{(k)},\nabla_S\bar{E}(U^{(k)})\rangle$}{

$$\alpha^{(k)} = q\alpha^{(k)};$$

Update $$U^{(k+1)} = \textup{ortho}(U^{(k)}, \delta U^{(k)}, \alpha^{(k)});$$
}

Calculate \begin{eqnarray*}
\beta^{(k)}&=&\frac{\|\nabla_S\bar{E}(U^{(k+1)})\|_F^2}{\|\nabla_S\bar{E}(U^{(k)})\|_F^2}; \\
\delta U^{(k+1)}&=&-\nabla_S\bar{E}(U^{(k+1)})+\beta^{(k)}(I_N-U^{(k+1)}{U^{(k+1)}}^T)\delta U^{(k)};
\end{eqnarray*}

Let $k=k+1$;
}

Return $\tilde{U} = U^{(k)}$ and $D = (I_N-U_nU_n^T)U^{(k)}.$
\end{algorithm}

The convergence of Algorithm \ref{CG} can be guaranteed under a mild assumption. If Assumption \ref{assum_lip} holds true, then
\begin{eqnarray*}
\|\nabla_S \bar{E}_n(U)\|_F &\le& 2\|\nabla\bar{E}_n(U)\|_F \\
&=&2\|\nabla_G E(U_n)+ \nabla_G^2 E(U_n)[(I_N-U_nU_n^T)U]\|_F \\
&\le& 2(C_0+C\|U\|_F) = 2(C_0+C\sqrt{N})
\end{eqnarray*}
is bounded and we are able to prove the convergence of Algorithm \ref{CG} by the similar strategy as the proof of Theorem 2.7 in \cite{DZZ}. Here, we state the convergence result without proof.
\begin{theorem}
Let Assumptions \ref{assum_lip} hold true. Assume that operator $\textup{ortho}(U, D, t)$ used in Algorithm \ref{CG} is one of the retractions. If $U_n \in \mathcal{M}_{N_g}^N$, then for the sequence $\{U^{(k)}\}_{k\in\mathbb{N}_0}$ generated by Algorithm \ref{CG}, there holds either $$\|\nabla_S \bar{E}_n(U^{(k)})\|_F = 0$$ for some positive integer $k$ or
$$\displaystyle\liminf_{k\to\infty} \|\nabla_S \bar{E}_n(U^{(k)})\|_F = 0.$$
\end{theorem}

\begin{remark}
Since $\mathcal{M}^N_{N_g}\subset\mathbb{R}^{N_g\times N}$, we may solve unconstrained problem
\begin{equation}\label{miniprob3}
\min_{U\in \mathbb{R}^{N_g\times N}} \ \ \ \bar{E}_n(U),
\end{equation}
with the same $\bar{E}_n(U)$ mentioned above. It is easy to confirm that \eqref{miniprob2} and \eqref{miniprob3} have the same solution. We can use the unconstrained CG method to solve \eqref{miniprob3}.
It is observed that the unconstrained CG method deduce to the CG method for solving \eqref{newton-equation-n} as a linear system which is applied in \cite{ZBJ}.
However, the Euclidean gradient of $\bar{E}(U)$ is not bounded on $\mathbb{R}^{N_g\times N}$. Hence, the sequence generated by the unconstrained CG method may diverge.
Here, we have found that it is not necessary to solve \eqref{miniprob3} in such a large set and proposed a new algorithm which converges.
\end{remark}

\subsection{Adaptive step size strategy}
Apart from the search direction, another important issue in Newton method is the choice of the step size. In Algorithm \ref{Newton} and Algorithm \ref{CG}, an Armijo-type backtracking strategy is used. It is shown in \cite{DZZ} that the backtracking procedure is costly and may be avoided by the adaptive step size strategy proposed therein. Here, we can also apply an adaptive step size strategy which is stated in the following Algorithm \ref{adaptive}.
\begin{algorithm}\label{adaptive}
\caption{{\bf Adaptive step size strategy}~($U, D, t^{\text{initial}}, t_{\textup{min}}, \eta, \theta$)}
  Set $t=\min{(\max{(t^{\text{initial}}, t_{\textup{min}})}, \theta/\|D\|_{V^N})}$;

  Calculate estimator $$\zeta(t) = \frac{\langle \nabla_G E(U),D\rangle + \frac{t}{2}\nabla_G^2 E(U)[D, D]}{\langle \nabla_G E(U),D\rangle};$$

  \If {$\zeta(t)<\eta$}{

  Choose \begin{equation*}
t=
\begin{cases}
\min\left(-\frac{\langle \nabla_G E(U),D\rangle}{\nabla_G^2E
(U)[D, D]}, \frac{\theta}{\|D\|_F}\right), & \mbox{if
 $\nabla_G^2E(U)[D, D]>0$},\\
\frac{\theta}{\|D\|_F}, & \mbox{otherwise};
\end{cases}
\end{equation*}
}

Return $t$;
\end{algorithm}

By using such a step size choice, we present our Newton method with adaptive step size strategy as Algorithm \ref{ANewton}
\begin{algorithm}\label{ANewton}
\caption{Newton method with adaptive step size strategy}
 Give $\epsilon,  \gamma_1, \gamma_2, q\in (0,1), \eta \in (0,\frac{1}{2})$,
  initial data $U_0, \ s.t. \  U_0^TU_0 = I_N$,
  calculate gradient $\nabla_G E(U_0)$, let $n = 0$\;
 \While{$\| \nabla_G E(U_n) \|_F> \epsilon$}{
  Choose suitable $\sigma_n, \theta_n\in(0,1)$;

  Solve $D_n\in\mathcal{T}_{[U_n]}\mathcal{G}^N_{N_g}$
  by Algorithm \ref{CG}
  such that
  \begin{equation*}
\|\nabla^2_G E(U_n)[D_n]+\nabla_GE(U_n)\|_F\le\sigma_n\|\nabla_G E(U_n)\|_F;
\end{equation*}

Calculate step size $$t_n = \textup{{\bf Adaptive step size strategy}}(U_n, D_n, t_n^{\textup{init}}, 10^{-2}, \eta, \theta_n);$$

  Update
      $U_{n+1}=\textup{ortho}(U_n,D_n,t_n)$;

  Let $n=n+1$, calculate gradient $\nabla_G E(U_n)$\;
 }
\end{algorithm}

We see from the proof of Theorem \ref{thm-conv} that the convergence of Algorithm \ref{ANewton} can be derived from Theorem 3.7 of \cite{DZZ}. We state the theoretical result without proof as Theorem \ref{thm-conv2}.
\begin{theorem}\label{thm-conv2}
Let Assumptions \ref{assum_lip} and \ref{assum_hess} hold true. If $\textup{ortho}(U, D, t)$ in Algorithm \ref{ANewton} is chosen to satisfy Assumption \ref{diff}, $[U_0] \in B([U^\ast], \delta_2) \cap \mathcal{L}$, then there exist sequences $\{\sigma_n\}_{n\in \mathbb{N}_0}\subset(0,1)$ and $\{\theta_n\}_{n\in \mathbb{N}_0}\subset(0,1)$ such that for the sequence $\{U_n\}_{n\in \mathbb{N}_0}$  generated by Algorithm \ref{ANewton}, there holds either $\|\nabla_G E(U_n)\|_F = 0$ for some positive integer $n$ or
\begin{equation*}
\lim_{n\to\infty} \| \nabla_G E(U_n)\|_F  = 0.
\end{equation*}
\end{theorem}

Here, the sequences $\{\sigma_n\}_{n\in \mathbb{N}_0}$ and $\{\theta_n\}_{n\in \mathbb{N}_0}$ can be chosen such that
\begin{eqnarray}\label{parameters}
\sigma_n &\le& \frac{\nu_1}{\nu_1+2C}, \\
\theta_n &=& \sup\{\tilde{\theta}_n: E(\textup{ortho}(U_n,D_n,t))-E(U_n)-t\langle\nabla_G E(U_n),D_n\rangle \notag \\
&&-\frac{t^2}{2}\nabla_G^2 E(U_n)[D_n,D_n]\le-\frac{\eta t\langle\nabla_G E(U_n),D_n\rangle}{2}, \forall t\le\frac{\tilde{\theta}_n}{\|D_n\|_F}\}, \notag
\end{eqnarray}
which inspires us to choose $\sigma_n$ to be a fixed constant being independent of $n$ and to choose $\theta_n$ differently for each $n$ in our numerical experiments.

\section{Numerical experiments} \label{sec-ne}
We report and analyze several numerical results in this section. We implement Algorithm \ref{Newton} (Newton-QR, with search directions given by Algorithm \ref{CG}) and Algorithm \ref{ANewton} (Newton-QR-A) based on package
Octopus\footnote[1]{Octopus:www.tddft.org/programs/octopus.} (version 4.0.1). All our numerical experiments are carried out on LSSC-IV in the State Key Laboratory of
Scientific and Engineering Computing of Chinese Academy of Sciences. In our simulation, the LDA exchange-correlation potential \cite{PeZu} is chosen to approximate $v_{xc}(\rho)$  and the Troullier-Martins
norm conserving pseudopotential \cite{TrMa} is used. The initial guess of the orbitals is generated by Linear Combination of the Atomic Orbits (LCAO) method.

Our examples include several typical molecular systems: benzene ($C_6H_6$), aspirin ($C_9H_8O_4)$,
fullerene ($C_{60}$), alanine chain $(C_{33}H_{11}O_{11}N_{11})$, carbon
nano-tube ($C_{120}$), and two carbon clusters $C_{1015}H_{460}$ and $C_{1419}H_{556}$.
We compare our results with those obtained by the conjugate gradient method proposed recently in
\cite{DLZZ} and we choose CG-QR algorithm, the algorithm that performs best
in \cite{DLZZ}, for comparison in our paper.

In our numerical experiments, we also use QR strategy (see Appendix \ref{app_retraction}) as the retraction and we set $\eta = \gamma_2 =$ 1e-4 which is recommended in \cite{NW}, and set $q=0.5, \gamma_1 = 0.1$. For all the systems except $C_{1015}H_{460}$ and $C_{1419}H_{556}$, tolerance $\epsilon$ is chosen to be $1$e$-12$, and for those two relatively large systems, $\epsilon = 1$e$-11$. Besides, we have found that the cost for solving Newton direction is expensive. To balance the accuracy of inner iteration and the total computational cost in our experiments, we set $\sigma_n = 0.4$ and terminate if the number of inner iteration reaches 3. We see from \eqref{taylor1} that in the formula of $\theta_n$ in \eqref{parameters}, $$E(\textup{ortho}(U_n,D_n,t))-E(U_n)-t\langle\nabla_G E(U_n),D_n\rangle-\frac{t^2}{2}\nabla_G^2 E(U_n)[D_n,D_n] = o(t^2\|D_n\|_F^2).$$
Hence, we may approximately choose
$$\theta_n = \big(\frac{-\eta\langle\nabla_GE(U_n), D_n\rangle}{\|D_n\|_F}\big)^{\frac{1}{1+\alpha}}, ~\alpha\in[0,1]$$
in our experiments.

We obtain from \cite{DLZZ} that the Grassmann Hessian of the Kohn-Sham total energy functional E(U) can be approximate by a part of itself, that is,
\begin{equation}\label{hes-pra}
\nabla_G^2E(U)[D_1, D_2] \approx \text{tr} ( D_2^T
\mathcal{H}(U)D_1) - \text{tr}( D_2^TD_1\Sigma), \forall U\in\mathcal{M}_{N_g}^N, D\in\mathcal{T}_{[U]}\mathcal{G}_{N_g}^N.
\end{equation}
As is pointed out in \cite{DLZZ}, the approximated
Hessian is very closed to the exact Hessian. Hence, we use the approximated Hessian \eqref{hes-pra} other than the exact Hessian \eqref{hes} in our experiments. The detailed numerical results are listed in Table \ref{t1}, where ``iter" means the number of iterations required to terminate the algorithm, $\| \nabla_G E\|_F$ forms the norm of the gradient when the algorithm terminates, ``wall clock time" is the total wall clock time spent to converge.

\begin{center}
\begin{table}[!htbp]
\caption{The numerical results for systems with different sizes obtained by different algorithms.}
\label{t1}
\begin{center}
{\small
\begin{tabular}{|c| c c c c|}
\hline
algorithm & energy (a.u.) & iter &  $\| \nabla_G E\|_F $ & wall clock time (s)\\
\hline
\multicolumn{5}{|c|}{benzene($C_6H_6) \ \ \  N_g = 102705 \ \ \  N = 15 \ \ \  cores = 8$} \\
\hline
CG-QR   & -3.74246025E+01 & 251  & 9.01E-13 & 12.58   \\
Newton-QR   & -3.74246025E+01 & 120  & 8.23E-13 & 13.30   \\
Newton-QR-A   & -3.74246025E+01 & 90  & 4.77E-13 & 8.63   \\
\hline
\multicolumn{5}{|c|}{aspirin($C_9H_8O_4) \ \ \  N_g = 133828 \ \ \  N = 34 \ \ \  cores = 16$} \\
\hline
CG-QR   & -1.20214764E+02 & 246  & 9.21E-13 & 29.21   \\
Newton-QR   & -1.20214764E+02 & 126  & 9.99E-13 & 40.28   \\
Newton-QR-A   & -1.20214764E+02 & 90 & 8.14E-13 & 27.09   \\
\hline
\multicolumn{5}{|c|}{$C_{60} \ \ \ N_g = 191805  \ \ \  N = 120 \ \ \  cores = 16$} \\
\hline
CG-QR   & -3.42875137E+02 & 391  & 9.45E-13 & 489.00   \\
Newton-QR   & -3.42875137E+02 & 196  & 9.49E-13 & 611.66   \\
Newton-QR-A   & -3.42875137E+02 & 149  & 6.68E-13 & 358.10   \\
\hline
\multicolumn{5}{|c|}{alanine chain$(C_{33}H_{11}O_{11}N_{11})\ \ \  N_g = 293725 \ \ \  N = 132
\ \ \  cores = 32$} \\
\hline
CG-QR   & -4.78562217E+02 & 2100  & 9.98E-13 & 2789.83   \\
Newton-QR   & -4.78562217E+02 & 1254  & 8.23E-13 & 5131.73   \\
Newton-QR-A   & -4.78562217E+02 & 1022  & 9.03E-13 & 2718.46   \\
\hline
\multicolumn{5}{|c|}{$C_{120} \ \ \ N_g = 354093 \ \ \  N = 240 \ \ \  cores = 32$} \\
\hline
CG-QR   & -6.84467048E+02 & 3517  & 9.90E-13 & 12976.96   \\
Newton-QR   & -6.84467048E+02 & 1806  & 9.99E-13 & 27247.76   \\
Newton-QR-A   & -6.84467048E+02 & 1291  & 9.83E-13 & 11402.85   \\
\hline
\multicolumn{5}{|c|}{$C_{1015}H_{460} \ \ \ N_g = 1462257 \ \ \  N = 2260 \ \ \  cores = 256$} \\
\hline
CG-QR   & -6.06369982E+03 & 266        & 9.17E-12 & 299047.84 \\
Newton-QR   & -6.06369982E+03 & 202        & 8.10E-12 &  446765.25\\
Newton-QR-A   & -6.06369982E+03 & 114        & 9.54E-12 & 208441.11 \\
\hline
\multicolumn{5}{|c|}{$C_{1419}H_{556} \ \ \ N_g = 1828847 \ \ \  N = 3116 \ \ \  cores = 320$} \\
\hline
CG-QR   & -8.43085432E+03 & 272         & 9.71E-12 & 722678.98 \\
Newton-QR   & -8.43085432E+03 &   178       & 8.30E-12 & 876253.24 \\
Newton-QR-A   & -8.43085432E+03 &  139    & 9.82E-12 & 584067.42 \\
\hline
\end{tabular}}
\end{center}
\end{table}
\end{center}

We see from Table \ref{t1} that the Newton method converges to solutions with desired accuracy within less iterations compared with the CG method. However, it still needs much computational time than the CG method because determining the Newton search direction is more costly than obtaining the CG direction and the backtracking procedure is also a very expensive part in Newton method. By using the adaptive step size strategy, we see that both the number of iterations and computational time are reduced, which make the Newton method with adaptive step size strategy to be more efficient than the CG method.

We also show the convergence curves for several systems in Figure \ref{fig-c9h8o4-resi}-\ref{fig-c1015-resi} to illustrate the advantages of Newton method more clear.

\begin{figure}
\centering
\caption{Convergence curves for $\| \nabla_G E\|_F $ obtained by
 different algorithms for $C_9H_8O_4$.}
\includegraphics[width=0.8\textwidth]{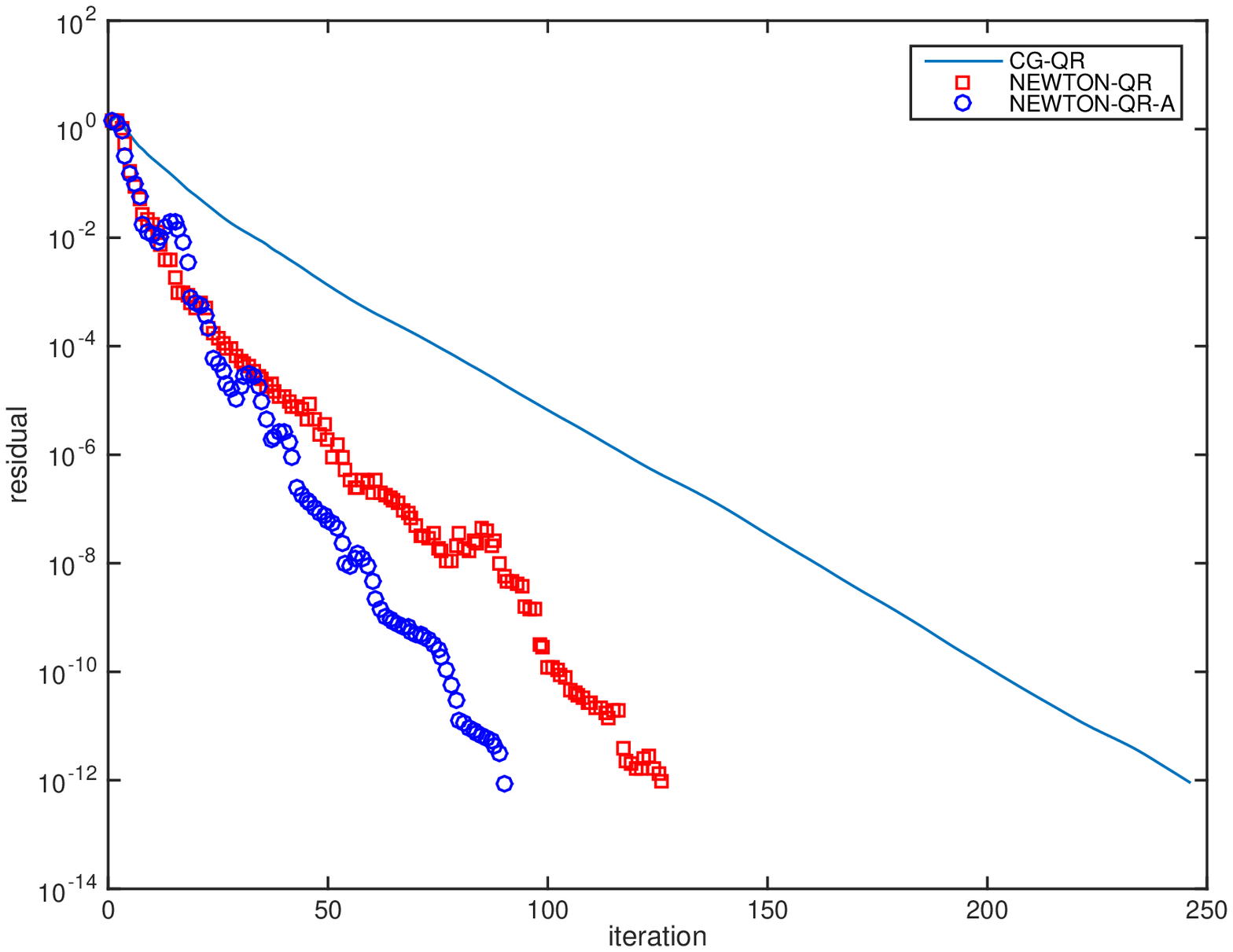} \\
\label{fig-c9h8o4-resi}
\end{figure}

\begin{figure}
\centering
\caption{Convergence curves for $\| \nabla_G E\|_F $ obtained by
 different algorithms for $C_{60}$.}
\includegraphics[width=0.8\textwidth]{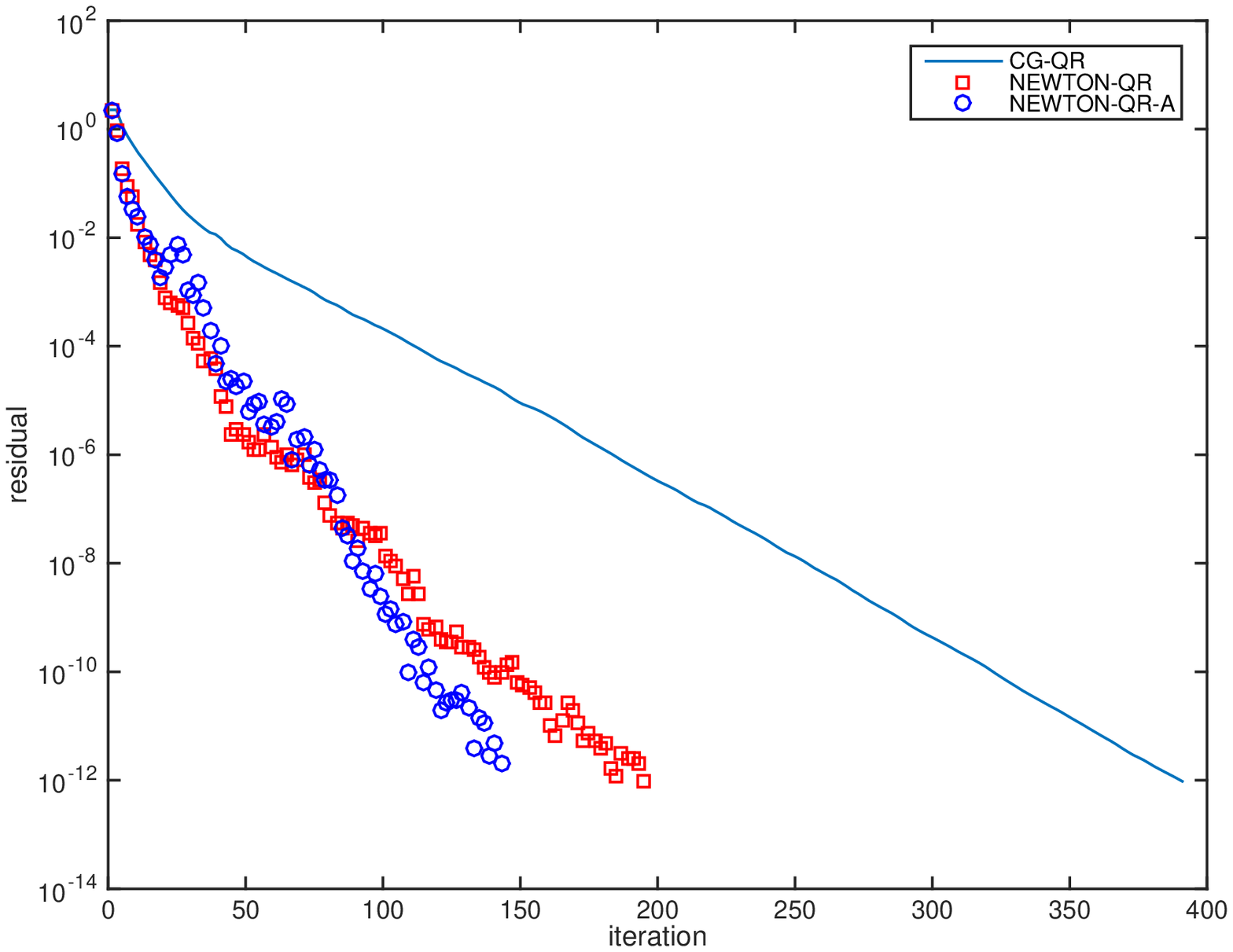} \\
\end{figure}

\begin{figure}
\centering
\caption{Convergence curves for $\| \nabla_G E\|_F $ obtained by
 different algorithms for alanine.}
\includegraphics[width=0.8\textwidth]{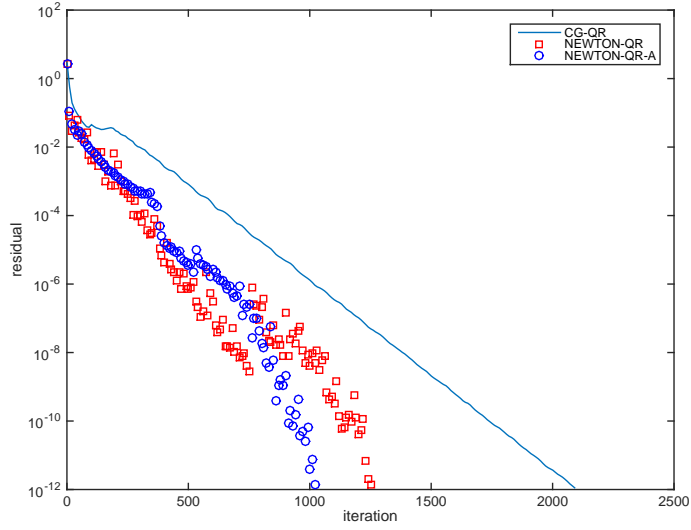} \\
\end{figure}

\begin{figure}
\centering
\caption{Convergence curves for $\| \nabla_G E\|_F $ obtained by
 different algorithms for $C_{120}$.}
\includegraphics[width=0.8\textwidth]{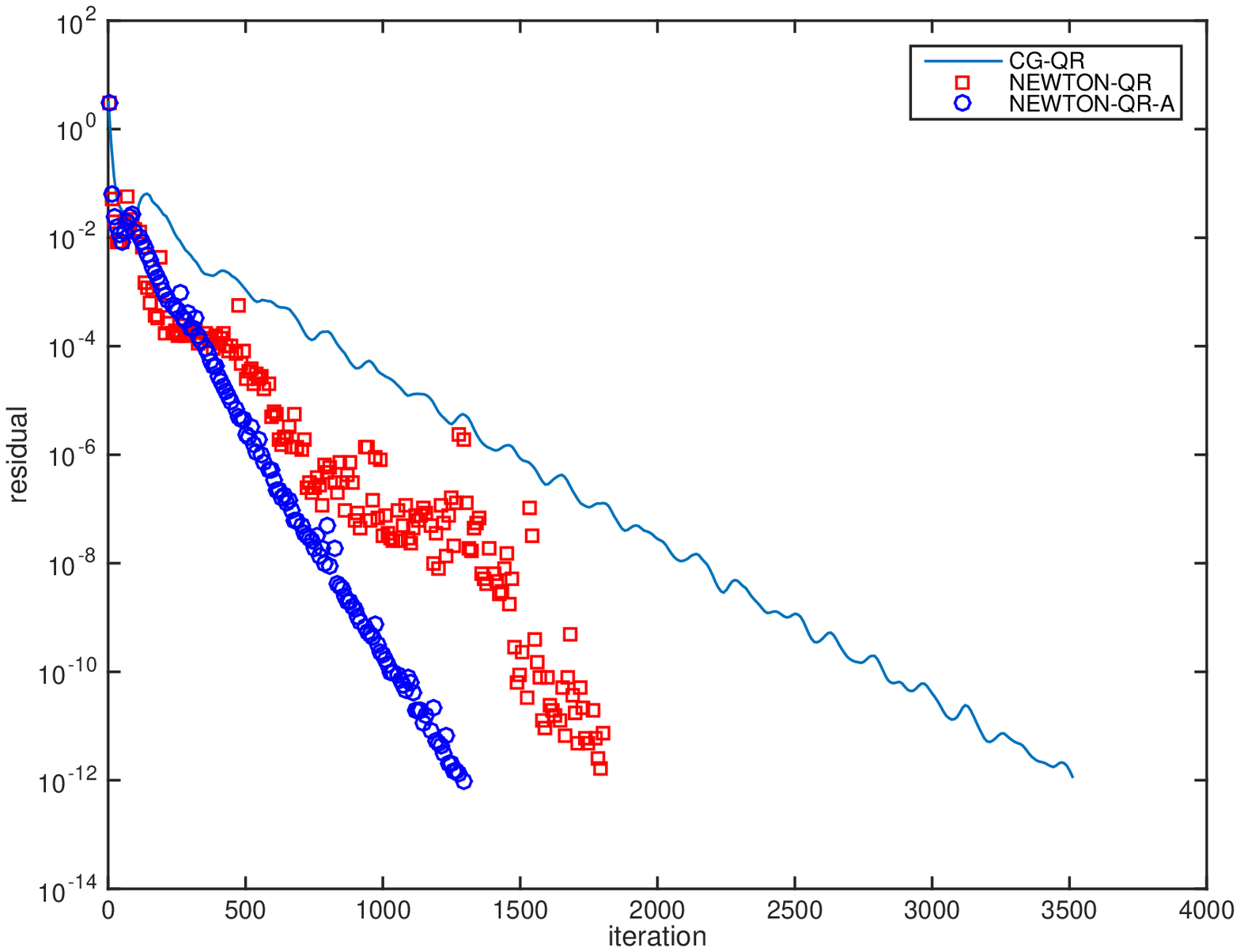} \\
\end{figure}

\begin{figure}
\centering
\caption{Convergence curves for $\| \nabla_G E\|_F $ obtained by
 different algorithms for $C_{1015}H_{460}$.}
\includegraphics[width=0.8\textwidth]{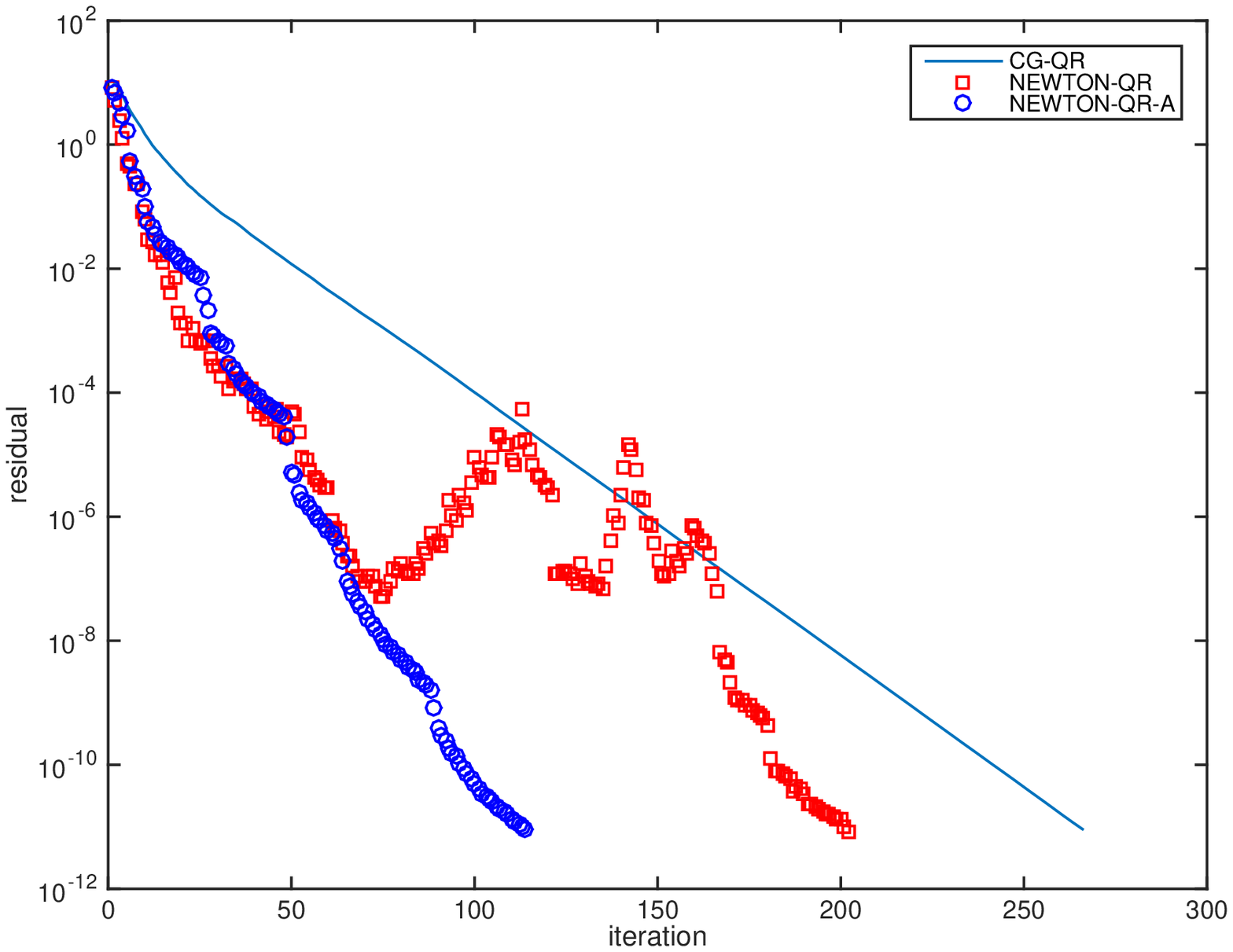} \\
\label{fig-c1015-resi}
\end{figure}


As is shown in Figure \ref{fig-c9h8o4-resi}-\ref{fig-c1015-resi}, the iteration sequences generated by Newton methods converge to the critical point more rapidly than the CG method. In particular, the Newton method with adaptive step size strategy converges much more stable and quickly than that with backtracked step sizes. Consequently, the Newton method equipped with the adaptive step size strategy is recommended.

\section{Concluding remarks} \label{sec-cln}
In this paper, we have studied the Newton methods for electronic structure calculations and shown the local convergence and convergence rate of the backtracking-based Newton method when the Newton directions are solved accurate enough. We have, in further, investigated the search direction solver which enables us to obtain satisfactory search directions during the iterations. Besides, we have applied an adaptive step size strategy \cite{DZZ}  to the Newton method which both maintains the convergence of the Newton method and makes the Newton method more practical in implementation. We have also reported several numerical experiments for different type of systems to show
that our methods are more efficient than the existing CG method. When comparing our methods themselves, we conclude that the Newton method with adaptive step size strategy performs better than that with the classic backtracking-based algorithm. Consequently, the Newton method with adaptive step size strategy is highly recommended.

We would like to mention that if there is a faster algorithm for solving \eqref{newton-equation}, which is usually
considered as ill-conditioned problem with large scale, then the efficiency of the Newton method will be further enhanced. This is indeed our ongoing work.

\Appendix
\renewcommand{\appendixname}{Appendix~\Alph{section}}
\section{Discussions on retractions}\label{app_retraction}
There are several orthogonality preserving strategies such as QR strategy, polar decomposition(PD) strategy and the so called Wen-Yin(WY) strategy (c.f. \cite{AbMaSe, WY, ZZWZ}). For these three strategies, the specific forms of $\text{ortho}(U,D,t)$ are as follows:

\begin{itemize}
      \item{for WY:}\begin{eqnarray*}
       \text{ortho}_{\textup{WY}}(U,D,t) &=& U + t D\Big(I_N+\frac{t^2}{4}
       D^TD  \Big)^{-1} \\
       && - \frac{t^2}{2}U\Big(I_N+\frac{t^2}{4}
         D^TD  \Big)^{-1}  (D^TD)  ;
      \end{eqnarray*}
      \item{for QR:}
      \begin{equation*}
          \text{ortho}_{\textup{QR}}(U,D,t) = (U + t D)L^{-T},
       \end{equation*}
       where $L$ is the lower triangular matrix such that
       \begin{equation*}
         LL^T = I_N + t^2
        D^T D   ;
       \end{equation*}
       \item{for PD:}\begin{equation*}
         \text{ortho}_{\textup{PD}}(U,D,t)=(U + t D)\big(I_N + t^2
        D^T D  \big)^{-\frac{1}{2}}.
      \end{equation*}\\
\end{itemize}

In \cite{JD}, Jiang and Dai has proposed a framework for designing orthogonality preserving operators as well as some specific schemes. It is worth mentioning that not all of the schemes therein are retractions.

Here, we would like to present a new class of retractions which are different from the existing approaches. We named our new strategies as Geodesic Approximation(GA) strategies.

If we let
\begin{equation} \label{antisy}
\mathcal{W} = DU^T - UD^T,
\end{equation}
then there holds $\mathcal{W}U=D$ and $\|\mathcal{W}\|_F\le2\|D\|_F$.

Since the geodesic \eqref{geodesic3} is the unique solution of
\begin{equation}
\left\{
\begin{aligned}
\overset{..}\Gamma(t) + \Gamma(t)\big(\overset{.}\Gamma(t)^T\overset{.}\Gamma(t)\big)& = 0, \\
\Gamma(0) & = U, \\
\overset{.}\Gamma(0) & = D,
\end{aligned}
\right.
\end{equation}
we see that the geodesic strategy can be expressed as
$$\textup{ortho}_{\textup{geo}}(U,D,t) = \Gamma(t) = e^{t\mathcal{W}}U,$$
where $e^{t\mathcal{W}}$ stands for the exponential mapping.

Though the geodesic method is the most straight forward and important strategy on solving manifold minimization problems, it is difficult to give the explicit form of the exponential mapping. Therefore, some approximations are required. Our GA strategies are then defined as
\begin{equation}\label{Utilde}
\textup{ortho}_{GA}(U, D, t)=\Big(I - \frac{t}{2} \mathcal{W} + P(t, \mathcal{W})^T\Big)^{-1} \Big(I + \frac{t}{2}
\mathcal{W} + P(t, \mathcal{W})\Big)U,
\end{equation}
where $P(t, \mathcal{W})$ satisfies the following conditions:
\begin{itemize}
\item $P(t, \mathcal{W})$ is commutative to $\mathcal{W}$ and $P(t, \mathcal{W})^T$, that is,
\begin{eqnarray}\label{commute}
    P(t, \mathcal{W})\mathcal{W}&=&\mathcal{W}P(t, \mathcal{W}),   \\
    P(t, \mathcal{W})P(t, \mathcal{W})^T &=& P(t, \mathcal{W})^TP(t, \mathcal{W}), \ \forall t\in\mathbb{R}. \notag
\end{eqnarray}
\item \begin{equation}\label{order}
\|P(t, \mathcal{W})\|_F=O(t^2\|D\|_F^2)
\end{equation} as $t\|D\|_F$ close to 0.
\item \begin{equation}\label{order2}
\|P(t, \mathcal{W})-P(t, \mathcal{W})^T\|_F=O(t^3\|D\|_F^3)
\end{equation} as $t\|D\|_F$ close to 0.
\end{itemize}

The following lemmas show that such a $\textup{ortho}_{GA}(U,D,t)$ is retraction.
\begin{lemma}
Suppose $\textup{ortho}_{GA}(U,D,t)$ is defined as \eqref{Utilde}, then $$\textup{ortho}_{GA}(U,D,t)^T\textup{ortho}_{GA}(U,D,t)=I_N, \ \ \forall t\in \mathbb{R}.$$
\end{lemma}
\begin{proof}
We have
\begin{eqnarray*}
\textup{ortho}_{GA}(U,D,t)^T\textup{ortho}_{GA}(U,D,t)&=&U^T\Big(I - \frac{t}{2}
\mathcal{W} + P(t, \mathcal{W})^T\Big)\Big(I + \frac{t}{2}
\mathcal{W} + P(t, \mathcal{W})\Big)^{-1} \\
&&\Big(I - \frac{t}{2} \mathcal{W} + P(t, \mathcal{W})^T\Big)^{-1} \Big(I + \frac{t}{2}
\mathcal{W} + P(t, \mathcal{W})\Big)U \\
&=&U^TU = I_N,
\end{eqnarray*}
in which \eqref{commute} is used. 
\end{proof}

\begin{lemma}
 Suppose $\textup{ortho}_{GA}(U,D,t)$ is defined as \eqref{Utilde}, then
 \begin{eqnarray}\label{prop}
 \textup{ortho}_{GA}(U,D,0)&=&U, \\
 \frac{\partial}{\partial t}\textup{ortho}_{GA}(U,D,0)&=&D.
 \end{eqnarray}
\end{lemma}
\begin{proof}
For any fixed $D\in\mathcal{T}_{[U]}\mathcal{G}_{N_g}^N$, we obtain from \eqref{order} that
\begin{equation}
\lim_{t\to 0}  \ \ \frac{\|P(t, \mathcal{W})\|_F}{t}=0,
\end{equation}
which implies $ P(0, \mathcal{W})=0$ and $\textup{ortho}_{GA}(U,D,0)=U$ is a straight forward result. Besides,
\begin{equation}\label{derivativezero}
  \|\frac{\partial}{\partial t}P(0, \mathcal{W})\|_F=\lim_{t\to 0}  \ \ \frac{\|P(t, \mathcal{W})-P(0, \mathcal{W})\|_F}{t}=0.
\end{equation}
Hence, $\frac{\partial}{\partial t}P(0, \mathcal{W})=0$. In addition, there holds from \eqref{Utilde} that
\begin{equation*}
\Big(I - \frac{t}{2} \mathcal{W} + P(t, \mathcal{W})^T\Big) \textup{ortho}_{GA}(U,D,t)= \Big(I + \frac{t}{2}\mathcal{W} + P(t, \mathcal{W})\Big)U.
\end{equation*}
Taking the derivative of both side with respect to $t$ gives that
\begin{eqnarray}\label{deriv}
&&\Big(-\frac{1}{2} \mathcal{W} - \frac{\partial}{\partial t}P(t, \mathcal{W})^T\Big) \textup{ortho}_{GA}(U,D,t) \\
&+& \Big(I - \frac{t}{2} \mathcal{W} + P(t, \mathcal{W})^T\Big) \frac{\partial}{\partial t}\textup{ortho}_{GA}{dt}(U,D,t)= \Big(\frac{1}{2}\mathcal{W} + \frac{\partial}{\partial t}P(t, \mathcal{W})\Big)U. \notag
\end{eqnarray}
Let $t=0$ in \eqref{deriv}, we have
\begin{equation*}
-\frac{1}{2}\mathcal{W}U+\frac{\partial}{\partial t}\textup{ortho}_{GA}{dt}(U,D,0)=\frac{1}{2}\mathcal{W}U,
\end{equation*}
or equivalently,
\begin{equation*}
\frac{\partial}{\partial t}\textup{ortho}_{GA}{dt}(U,D,0)=\mathcal{W}U=D.
\end{equation*}
The definition of $\mathcal{W}$ is used in the last equality which completes the proof.
\end{proof}

There are infinite number of GA approaches. We present some examples here.

\newtheorem{exam}[theorem]{Example}
\begin{exam}
Choose $P(t, \mathcal{ W})=\bf{0}$, for which \eqref{commute}, \eqref{order} and \eqref{order2} are satisfied. Then the GA strategy degenerates to the WY strategy, that is,
$$\textup{ortho}_{GA}(U,D,t)=\Big(I - \frac{t}{2} \mathcal{W}\Big)^{-1} \Big(I + \frac{t}{2}
\mathcal{W}\Big)U=\textup{ortho}_{WY}(U,D,t).$$
We refer to \cite{DLZZ, WY} for more details about the WY strategy.
\end{exam}
\begin{exam}\label{poly}
Choose $$P(t, \mathcal{W})=t^2\sum_{i=0}^kc_i(t\mathcal{W})^i\mathcal{W}^2,$$ with $c_i$ be some real coefficients, then $P(t, \mathcal{W})=o(t^2\|W\|^2)=o(t^2\|D\|^2)$ and
\begin{eqnarray*}
\|P(t, \mathcal{W})-P(t, \mathcal{W})^T\|_F &=& t^2\|\sum_{i=1}^{\lfloor\frac{k+1}{2}\rfloor}c_{2i-1}(t\mathcal{W})^{2i-1}\|_F\|\mathcal{W}\|_F^2 \\
&=&o(t^3\|\mathcal{W}\|_F^3)=o(t^3\|D\|_F^3).
\end{eqnarray*}
In addition, we see that $P(t, \mathcal{W})$ is a polynomial of $\mathcal{W}$ so that it is commutative to $\mathcal{W}$ and $P(t, \mathcal{W})^T$.
As a result, $$P(t, \mathcal{W})=t^2\sum_{i=0}^kc_i(t\mathcal{W})^i\mathcal{W}^2$$ satisfies \eqref{commute}, \eqref{order} and \eqref{order2}.
\end{exam}

We observe that the WY strategy is nothing but the $(1,1)$ type Pad\'{e} approximation of the exponential mapping. We can obviously spread this idea to $(k,k)$ type Pad\'{e} approximation, which is more closed to the geodesic than the WY strategy, for instance,
\begin{eqnarray*}
&\textup{ortho}_{(2,2)}(U,D,t)&=\Big(I - \frac{t}{2} \mathcal{W} + \frac{t^2}{12}\mathcal{W}^2\Big)^{-1}\Big(I + \frac{t}{2}
\mathcal{W}+ \frac{t^2}{12}\mathcal{W}^2\Big)U \\
&\textup{ortho}_{(3,3)}(U,D,t)&=\Big(I - \frac{t}{2} \mathcal{W} + \frac{t^2}{10}\mathcal{W}^2 - \frac{t^3}{120}\mathcal{W}^3\Big)^{-1}\Big(I + \frac{t}{2}
\mathcal{W}+ \frac{t^2}{10}\mathcal{W}^2 + \frac{t^3}{120}\mathcal{W}^3\Big)U \\
&\cdots&
\end{eqnarray*}

They are all included in Example \ref{poly}.
\begin{remark}
It can be check that the similar formula which based on the $(p,q)$ type Pad\'{e} approximation with $p\neq q$, does not preserve the orthogonality.
\end{remark}

We last show that the geodesic $\textup{ortho}_{geo}(U,D,t)$ can also be expressed as the form of \eqref{Utilde}.
\begin{exam}
Choosing $P(t, \mathcal{W})=e^{\frac{t\mathcal{W}}{2}}-I-\frac{t\mathcal{W}}{2}$, we can check by the properties of exponential mapping that \eqref{commute}, \eqref{order} and \eqref{order2} are satisfied for this $P(t, \mathcal{W})$.
Taking $P(t, \mathcal{W})$ into \eqref{Utilde}, we have
\begin{eqnarray*}
\textup{ortho}_{GA}(U,D,t)&=&\Big(I - \frac{t}{2} \mathcal{W} + e^{-\frac{t\mathcal{W}}{2}}-I+\frac{t\mathcal{W}}{2}\Big)^{-1} \\
&&\Big(I + \frac{t}{2}\mathcal{W} + e^{\frac{t\mathcal{W}}{2}}-I-\frac{t\mathcal{W}}{2}\Big)U \\
&=&e^{\frac{t\mathcal{W}}{2}}e^{\frac{t\mathcal{W}}{2}}U = e^{t\mathcal{W}}U=\textup{ortho}_{geo}(U,D,t).
\end{eqnarray*}
\end{exam}

This indicates that the geodesic method itself is also contained in our proposed GA methods.

We then show that all the mentioned retractions satisfy Assumption \ref{lema-smdif}.
\begin{theorem}
Let $\textup{ortho}_1(U,D,t)$ and $\textup{ortho}_2(U,D,t)$ to be two arbitrary retractions of the QR, PD, WY or GA strategy. Then for any $D\in\mathcal{T}_{[U]}\mathcal{G}^N_{N_g}$, with $t\|D\|_F<1$, there holds
\begin{equation*}
E(\textup{ortho}_1(U,D,t))-E(\textup{ortho}_2(U,D,t)) = O(t^3\|D\|_F^3).
\end{equation*}
\end{theorem}
\begin{proof}
For simplicity, we denote the first and second order derivative of the retraction $\textup{ortho}(U,D,t)$ with respective to $t$ by $\dot{\textup{ortho}}(U,D,t)$ and $\ddot{\textup{ortho}}(U,D,t)$, respectively. We first show that for the WY and PD strategy, there exist a constant $M_1>0$ such that
\begin{equation}\label{2ndder}
\|\ddot{\textup{ortho}}(U,D,t)+UD^TD\|_F\leq M_1t\|D\|_F^3, \forall (t,D)\in\mathbb{R}\times\mathcal{T}_{[U]}\mathcal{G}_{N_g}^N,t\|D\|_F\leq1.
\end{equation}
In fact, it can be calculated that
\begin{eqnarray*}
\ddot{\textup{ortho}}_{WY}(U,D,t)&=&\big(I-\frac{t\mathcal{W}}{2}\big)^{-1}\mathcal{W}\dot{\textup{ortho}}_{WY}(U,D,t), \\
\ddot{\textup{ortho}}_{PD}(U,D,t)&=&t^2\textup{ortho}_{PD}(U,D,t)(D^TD)^2(I+t^2D^TD)^{-2} \\
&&-2t\dot{\textup{ortho}}_{PD}(U,D,t)(D^TD)(I+t^2D^TD)^{-1} \\
&&-\textup{ortho}_{PD}(U,D,t)(D^TD)(I+t^2D^TD)^{-1},
\end{eqnarray*}
where $\mathcal{W}=DU^T-UD^T$. Thus, $\mathcal{W}D= -UD^TD$. Noting that $\|(I-\frac{t\mathcal{W}}{2})^{-1}\|_2\leq 1$ and \eqref{prop4.2}, we have
\begin{eqnarray*}
&&\|\ddot{\textup{ortho}}_{WY}(U,D,t)+UD^TD\|_F \\
&\leq&\|(I-\frac{t\mathcal{W}}{2})^{-1}\|_2\|\mathcal{W}\dot{\textup{ortho}}_{WY}(U,D,t)-(I-\frac{t\mathcal{W}}{2})\mathcal{W}D\|_F \\
&\leq&\|\mathcal{W}(\dot{\textup{ortho}}_{WY}(U,D,t)-\dot{\textup{ortho}}_{WY}(U,D,0))\|_F+\frac{t}{2}\|\mathcal{W}^2D\|_F \\
&\leq&C_2t\|\mathcal{W}\|_F\|D\|_F^2+\frac{t}{2}\|\mathcal{W}\|_F^2\|D\|_F = (\sqrt{2}C_2+1)t\|D\|_F^3,
\end{eqnarray*}
where the fact that $\|\mathcal{W}\|_F = \sqrt{2}\|D\|_F$ is used in the last equality. Similarly, we obtain by $\|(I+t^2D^TD)^{-k}\|_F \leq 1, \forall k\ge 0$ that
\begin{eqnarray*}
&&\|\ddot{\textup{ortho}}_{PD}(U,D,t)+UD^TD\|_F \\
&\leq&t^2\|D\|_F^4+2t\|(\dot{\textup{ortho}}_{PD}(U,D,t)-D)(D^TD)\|_F+2t\|D\|_F^3 \\
&&+\|(U-\dot{\textup{ortho}}_{PD}(U,D,t))(D^TD)\|_F+t^2\|U(D^TD)^2\|_F \\
&\leq&(2+C_1)t\|D\|_F^3+4t^2\|\|D\|_F^4\leq(6+C_1)t\|D\|_F^3,
\end{eqnarray*}
provided that $t\|D\|_F \le 1$.
Thus, $M_1$ can be chosen as $\max\{\sqrt{2}C_2+1, C_1+6\}$. By \eqref{2ndder}, we see that there exists $\xi\in(0,t)$ such that
\begin{eqnarray*}
\|\textup{ortho}_*(U,D,t)-U-tD+\frac{t^2}{2}UD^TD\|_F &=&\frac{t^2}{2}\|\ddot{\textup{ortho}}_*(U,D,\xi)+UD^TD\|_F \\
&\leq&\frac{t^2\xi M_1}{2}\|D\|_F^3\leq\frac{M_1}{2}t^3\|D\|_F^3.
\end{eqnarray*}
Here, the subscript * can be $WY$, $PD$. In further, we have
\begin{eqnarray*}
&&\|\textup{ortho}_{GA}(U,D,t)-\textup{ortho}_{WY}(U,D,t)\|_F \\
&=&\|\Big(I - \frac{t}{2} \mathcal{W} + P(t, \mathcal{W})^T\Big)^{-1}\Big(I + \frac{t}{2}\mathcal{W} + P(t, \mathcal{W})\Big)U - (I - \frac{t}{2} \mathcal{W})^{-1}(I + \frac{t}{2}\mathcal{W})U\|_F \\
&\leq&\|\Big(I - \frac{t}{2} \mathcal{W} + P(t, \mathcal{W})^T\Big)^{-1}\|_2\|P(t,\mathcal{W})U + P(t,\mathcal{W})^T(I - \frac{t}{2} \mathcal{W})^{-1}(I + \frac{t}{2}\mathcal{W})U\|_F \\
&\leq&\|\Big(P(t,\mathcal{W})-P(t,\mathcal{W})^T\Big)U\|_F+\|P(t,\mathcal{W})^T\|_F\|(I - \frac{t}{2} \mathcal{W})^{-1}(I + \frac{t}{2}\mathcal{W})U\|_F \\
&\leq&O(t^3\|D\|_F^3)+O(t^2\|D\|_F^2)\|(I - \frac{t}{2} \mathcal{W})^{-1}\|_2\|t\mathcal{W}U\|_F \\
&=&O(t^3\|D\|_F^3).
\end{eqnarray*}
Thus,
\begin{eqnarray}\label{GAappro}
&&\|\textup{ortho}_{GA}(U,D,t)-U-tD+\frac{t^2}{2}UD^TD\|_F \notag\\
&\leq&\|\textup{ortho}_{GA}(U,D,t)-\textup{ortho}_{WY}(U,D,t)\|_F \notag\\
&&+\|\textup{ortho}_{WY}(U,D,t)-U-tD+\frac{t^2}{2}UD^TD\|_F \notag\\
&=&O(t^3\|D\|_F^3).
\end{eqnarray}
As a result, there holds by \eqref{taylor1} and \eqref{der-bound} that
\begin{eqnarray*}
&&|E(\textup{ortho}_1(U,D,t))-E(\textup{ortho}_2(U,D,t))| \\
&\leq& C_0\textup{dist}_{geo}(\textup{ortho}_1(U,D,t),\textup{ortho}_2(U,D,t)) \\
&\leq&2C_0\textup{dist}_{F}(\textup{ortho}_1(U,D,t),\textup{ortho}_2(U,D,t)) \\
&\leq&2C_0\|\textup{ortho}_1(U,D,t)-\textup{ortho}_2(U,D,t)\|_F \\
&=&O(t^3\|D\|_F^3),
\end{eqnarray*}
where both $\textup{ortho}_1(U,D,t)$ and $\textup{ortho}_2(U,D,t)$ are one of WY, PD, or GA strategy.
As for QR strategy, we note that $[\textup{ortho}_{QR}(U,D,t)]=[\textup{ortho}_{PD}(U,D,t)]$ and completes our proof by
\begin{equation}
E(\textup{ortho}_{QR}(U,D,t))=E(\textup{ortho}_{PD}(U,D,t)).
\end{equation}
\end{proof}

We see from \eqref{GAappro} that the GA methods satisfy Assumption \ref{diff}, i.e., \eqref{prop4.1} and \eqref{prop4.2} hold for GA strategies.

\end{document}